\documentclass[a4paper,10pt]{article}
\usepackage{a4}
\usepackage{amsmath, amsfonts, amssymb, amsthm, amscd, bezier}
\usepackage[brazil,english]{babel}
\usepackage[T1]{fontenc}
\usepackage{graphicx}
\usepackage[all]{xy}
\usepackage{color}
\usepackage{fullpage}
\usepackage{cite}
\advance\oddsidemargin by -0.5cm \advance\textwidth by 1cm
\usepackage{mathrsfs}
\usepackage[utf8]{inputenc}	
\usepackage{comment}
\usepackage{indentfirst}
\usepackage{multirow}
\def\href{}

\usepackage[pagebackref=true]{hyperref}
\renewcommand*{\backref}[1]{}
\renewcommand*{\backrefalt}[4]{%
  \ifcase #1
    Not cited.
  \or
    (Cited on page #2.)
  \else
    (Cited on pages #2.)
  \fi
}
\PassOptionsToPackage{x11names}{xcolor}

\usepackage{tikz}
\usetikzlibrary{arrows,snakes,positioning}

\usepackage{lmodern}			
\usepackage[T1]{fontenc}		
\usepackage[utf8]{inputenc}		

\usepackage{color,xcolor}		
\usepackage{graphicx}			
\usepackage{microtype} 			

\usepackage[english,onelanguage]{algorithm2e}	
 \usepackage{amsbsy}			
 \usepackage{amscd}			
 \usepackage{amsfonts}			
 \usepackage{amsmath}			
 \usepackage{amssymb}			
 \usepackage{amstext}			
\usepackage{amsthm}			
\usepackage{hyperref}			
\usepackage{cleveref}			
\usepackage{dsfont}			
 \usepackage{ifthen}			
\usepackage{listings}           	
\usepackage{mathabx}			
 \usepackage{mathrsfs}			
\usepackage{listings}           	

\usepackage{soul} 

\definecolor{azulagua}{RGB}{11, 243, 249}
\definecolor{green}{RGB}{17, 138, 49}
\definecolor{purple}{RGB}{59, 38, 110}

\newtheorem{proposition}{Proposition}

\newtheorem{definition}{Definition}

\newtheorem{remark}{Remark}

\usepackage{float}

\begin{document}

\begin{center}

\textbf{  \LARGE  \bf {Conley Index of Limit Cycles in Filippov Vector Fields}
}

\vspace{0.55cm}

\hspace{0.3cm} 
L. Cândido \footnote{Financed by the Coordenação de Aperfeiçoamento de Pessoal de Nível Superior - Brasil (CAPES) - Finance Code 001}
\hspace{0.3cm} D. V. S. Lima \footnote{ Supported by FAPESP under grants 	2025/12435-9, 2022/16455-6. Supported by CAPES through the MATH-AmSud program, Grant No. 88881.179491/2025-01.} 
\hspace{0.3cm}K. A. de Rezende
\footnote{Partially supported by FAPESP under grant 2022/16455-6. Supported by CAPES through the MATH-AmSud program, Grant No. 88881.179491/2025-01. }

\end{center}

\begin{abstract}
Within the framework of multivalued dynamical systems, we compute the Conley index for a class of limit cycles in planar Filippov systems. We classify these cycles according to their interaction with the switching manifold, including crossing, sliding, and mixed regimes. Using the Conley index for multivalued admissible maps developed by Mrozek, we construct the appropriate index pairs and determine the corresponding indices for each class. These computations enlarge the class of Filippov invariant sets for which the Conley index is known and extend previous results on low-codimension singularities to periodic dynamics.
\end{abstract}

\noindent \textbf{Keywords:} Conley Index, Filippov Vector Fields, Limit Cycle.

\noindent \textbf{AMS subject classification:} 	37B30; 34C25; 34C05

\tableofcontents

\section{Introduction}

Limit cycles represent a class of fundamental recurrent phenomena in dynamical systems, playing a central role in the understanding of long-term behavior \cite{du2008bifurcation}, \cite{novaes2020study}. In the context of smooth vector fields, hyperbolic limit cycles are structurally stable objects. The stability properties and bifurcations of these cycles are of fundamental importance to numerous applications, encompassing a broad spectrum of disciplines, from classical mechanics to biological modeling, \cite{VanDerPol1926}, \cite{Winfree1967}.

In Filippov systems, however, the dynamics become considerably richer and more subtle due to the discontinuity of the vector field, \cite{filippov2013differential}. It is important to note that trajectories may interact with the switching manifold through crossing, sliding, or escaping motions. This interaction gives rise to new classes of periodic phenomena that have no analogue in the smooth setting. In particular, the phenomenon of "sliding cycles" and "hybrid cycles" has been observed, which exhibit a combination of smooth with sliding dynamics along the discontinuity set.

These nonsmooth periodic orbits are not merely mathematical curiosities; they arise naturally in a wide variety of applications, including mechanical systems with dry friction, electrical circuits with idealized switching components, and biological models with threshold or relay-type mechanisms, \cite{cristiano2017hopf}, \cite{leinebifurcations}, \cite{PILTZ}. Consequently, a comprehensive understanding of their existence, stability, and bifurcation mechanisms is imperative for both the theoretical development of nonsmooth dynamical systems and their applications. Despite their relevance, characterizing the global behavior of the associated dynamics remains a difficult task.

Conley index theory offers a natural approach to addressing this issue, \cite{conley1971isolated}, \cite{conley1978isolated}. The main advantage of this theory is its robustness under continuation; the index depends only on the local behavior around the isolated invariant set and persists under small perturbations. This makes it well-suited for nonsmooth systems, where trajectories are not uniquely defined and classical flow techniques fail.

Mrozek developed a fundamental extension of this theory by constructing the Conley index for multivalued admissible maps and dynamical systems, \cite{mrozek1990cohomological}. This framework is essential in the present context. Filippov systems can naturally be formulated as differential inclusions, whose solutions are generally non-unique and evolve according to set-valued dynamics. The multivalued Conley-Mrozek framework retains the key properties of the classical theory, such as homotopy invariance, continuation, and detectability of isolated invariant sets, while providing a rigorous setting that fully incorporates sliding motion and discontinuous evolution. This makes the Conley-Mrozek framework one of the few tools capable of treating nonsmooth dynamics in a genuinely global and topological way.

In the previous work \cite{candido2026conley}, we computed the Conley indices of several low-codimension singularities in planar Filippov systems. This demonstrated how the index captures local nonsmooth dynamics near the switching manifold.

In this paper, we turn to isolated Filippov limit cycles. We compute the Conley index for each class of limit cycle according to its interaction with the switching manifold. These results expand the class of Filippov invariant sets for which the Conley index is known and extend the previous singularity-based analysis to periodic orbits. This work provides a step toward a topological classification of invariant structures in nonsmooth systems.

This paper is organized as follows. In Section~\ref{preliminary} we provide a background material on  Filippov vector fields,
differential inclusions, and Conley index for multivalued flows. In Section~\ref{sec:CILC} we construct a Conley-Mrozek index pair for each class of Filippov limit cycles and 
calculate its Conley index. Finally, in Section~\ref{sec:finalremarks} we discuss future directions using Conley-Mrozek theory in the investigation of global dynamics in Filippov vector fields.

\section{Background}\label{preliminary}
Let $U\subset \mathbb{R}^n$ be a bounded connected open set and denote by $K=\overline{U}$ its closure. We denote by $C^r(K,\mathbb{R}^n)$ the space of $C^r$-smooth vector fields $X: K \rightarrow \mathbb{R}^n$ on $K$. Given a differentiable function $h:K \rightarrow \mathbb{R}$  such that $0$ is a regular value of $h$, we denote by  $\Omega^r_h(K,\mathbb{R}^n)$ the space of piecewise smooth vector fields $Z$ such that for  each $x \in K$
\begin{equation}\label{first}
    Z(x)=\left\{
    \begin{array}{ccc}
       X(x), & \text{if} & h(x)>0, \\
        Y(x), & \text{if} & h(x)<0, \\ 
    \end{array}
    \right.
\end{equation}
where $X,Y \in C^r(K,\mathbb{R}^n)$.
 A vector field of the form  \eqref{first} is usually denoted by  $Z=(X,Y)$, and it is referred to as a \textbf{Filippov system}. The  hypersurface $\Sigma=h^{-1}(0)$ is referred to as the \textbf{switching manifold of codimension 1}.

The space $C^r(K,\mathbb{R}^n)$ is ewdowed with the usual  $C^r$ topology, while $\Omega^r_h(K,\mathbb{R}^n)$ carries the correspoding product topology. Since the vector field may be discontinuous along $\Sigma$, solutions are understood in the sense of differential inclusions.

Recall that,  the \textbf{convex closure} of a subset  $A\subset \mathbb{R}^n$ is the intersection of all convex sets containing a $A$ and its denoted by  $Conv(A)$.

Given a Filippov system $Z=(X,Y) \in \Omega_h^r(K,\mathbb{R}^n)$, define  the function $$ f(x,\lambda)=\dfrac{1+\lambda}{2}X(x)+\dfrac{1-\lambda}{2}Y(x),$$ where $ \lambda \in [-1,1]. $ Then $Conv(\{X(x),Y(x)\})=\{f(x,\lambda); \lambda \in [-1,1] \}$.
Associated to $Z$, we define the multi-valued map\footnote{Throughout this paper, the symbol $\twoheadrightarrow$ is used to represent multi-valued maps.} $ \mathcal{F}_Z(x) \twoheadrightarrow \mathbb{R}^n$ by
\begin{equation}\label{Z}
    \mathcal{F}_Z(x)=\left\{
    \begin{array}{lcc}
      \{ X(x)\}, & \text{if} & h(x)>0, \\
       \{ Y(x)\}, & \text{if} & h(x)<0, \\ 
       Conv(\{X(x),Y(x)\}),  & \text{if} & h(x)=0.
    \end{array}
    \right.
\end{equation}
A curve $\varphi:(-t,t)\rightarrow K$ is said to be a \textbf{solution} of the Filippov system \eqref{first} if it is a solution of $\dot{x} \in \mathcal{F}_Z(x).$ Hence, the dynamics of a Filippov system can be studied through the \textbf{differential inclusion $\mathcal{F}_Z$} given in \eqref{Z}.
Conditions ensuring existence of solutions for such inclusions were established by Filippov in \cite{filippov2013differential}, and termed as basic conditions.

\begin{definition}
    A \textbf{multi-valued map $f:U \twoheadrightarrow V$} is a mapping which assigns to every $x \in U$ some subset $f(x) \subset V$.
\end{definition}
\begin{definition}
    Let $\mathcal{F}:U\subset \mathbb{R}^n\twoheadrightarrow \mathbb{R}^n$ be an upper-semicontinuous multi-valued map.  We say that  $\mathcal{F}$ satisfies the \textbf{basic conditions} if, for all $x_0 \in U$, the set $\mathcal{F}(x_0)$ is nonempty, compact and convex.
\end{definition}
If a differential inclusion satisfies the basic conditions, we say that it is a \textbf{basic differential inclusion}. If $\dot{x} \in \mathcal{F}_Z(x)$ is a Filippov system, then $\mathcal{F}_Z$ satisfies the basic conditions, see \cite{thieme2021conley}.

\subsection{Dynamics on the Switching Manifold}

 For each $p \in \Sigma$,  the \textbf{Lie derivative} of $h(p)$  in the direction of the vector  $X(p)$ is defined by $$Xh(p)=\langle \nabla h(p),X(p)\rangle.$$  Analousgously,  define $X^2h(p)=\langle \nabla Xh(p),X(p)\rangle$ and, recursively, $X^nh(p)=\langle \nabla X^{(n-1)}h(p),X(p)\rangle$.
Observe that $Xh(p)=0$ if and only if the vector $X(p)$ is tangent to the switching manifold $\Sigma$ at $p$. Such a point is called a \textbf{tangency point}.
Using the notation introduced in  \eqref{first}, the switching manifold $\Sigma$ can be decomposed into the following regions:
\begin{itemize}
\item $\Sigma^c:=\{p  \in \Sigma ;Xh(p)Yh(p) > 0\}$, called the \textit{crossing} region.
\item $\Sigma^e:=\{p \in \Sigma ; Xh(p)Yh(p) < 0 \mbox{ with } Xh(p) > 0 \mbox{ and } Yh(p) < 0\}$, called the \textit{escaping} region.
\item $\Sigma^s:=\{p \in \Sigma;Xh(p)Yh(p) < 0 \mbox{ with } Xh(p) < 0 \mbox{ and } Yh(p) > 0\}$, called the \textit{sliding} region.
\end{itemize}
In \autoref{regioes}, we illustrate these regions.
\begin{figure}[h!]
    \centering
    \includegraphics[width=9cm]{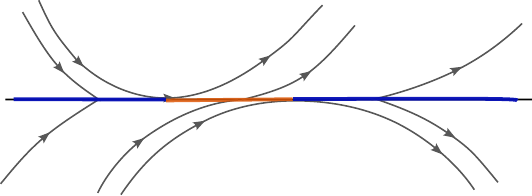}
    \put(-263,42){$\Sigma$}
    \put(-280,80){$X$}
    \put(-280,0){$Y$}
    \put(-178,43){$\bullet$}
    \put(-178,50){\small $x_1$}
    \put(-118,43){$\bullet$}
    \put(-118,34){\small $x_2$}
    \put(-240,34){\small \textcolor{blue}{$\Sigma^s$}}
    \put(-150,34){\small\textcolor{orange}{$\Sigma^c$}}
    \put(-40,34){\small\textcolor{blue}{$\Sigma^e$}}
    \caption{Sliding, crossing and escaping regions, $\Sigma^s$, $\Sigma^c$, $\Sigma^e$ respectively.}
        \label{regioes}
\end{figure}

The points $x_1$ and $x_2$ shown in \autoref{regioes} are tangency points. In particular, they belong to the boundaries of the regions $\Sigma^s$, $\Sigma^c$ and $\Sigma^e$, since $Xh(x_1)=0$ and $Yh(x_2)=0$.

If $p\in \Sigma^c$, then the trajectory passing through $p$ is obtained by concatenating the trajectories $\varphi_X$ and $\varphi_Y$, of the vector fields $X$ and $Y$, on $p$.  On the other hand, when $p\in \Sigma^s\cup \Sigma^e$, the dynamics on $\Sigma$ is described by an additional vector field, called the sliding vector field and defined below.

\begin{definition}\label{def3}
   Under the above conditions, for every $p\in \Sigma^e \cup \Sigma^s$, the \textbf{sliding vector field} associated to $Z=(X,Y)$, denoted by $\widetilde{Z}$, is defined by:
    \begin{equation}\label{campodeslizante}
        \widetilde{Z}(p)=\frac{Yh(p)X(p)-Xh(p)Y(p)}{Yh(p)-Xh(p)}.
    \end{equation}
\end{definition}  See \autoref{campodeslizantefig}.

\begin{figure}[h!]
    \centering
     \includegraphics[width=7.5cm]{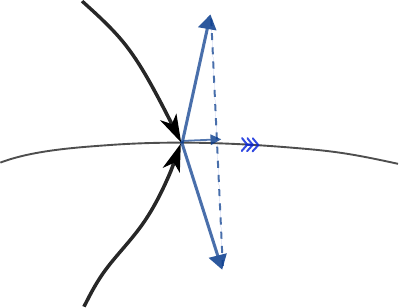}
     \put(-225,75){$\Sigma$}
     \put(-200,130){$X$}
     \put(-200,30){$Y$}
     \put(-118,85){$\bullet$}
     \put(-116,74){\small $p$}
     \put(-105,10){\small $Xh(p)$}
     \put(-110,160){\small $Yh(p)$}
    \caption{Sliding vector field defined on $p$.}
    \label{campodeslizantefig}
\end{figure}

Consequently, if $\varphi_{\widetilde{Z}}$ is a solution of $\dot{x}=\widetilde{Z}(x)$ with initial condition $x(0)=p$, then $\varphi_{\widetilde{Z}}$ is a solution of $\dot{x}\in \mathcal{F}_Z(x)$.

Observe that, if $p \in \Sigma^s$, then for $t<0$ sufficiently small one has $h(\varphi_X(p,t))>0$ and $h(\varphi_Y(p,t))<0$. In a similar way, if $p \in \Sigma^e$, then for $t>0$ sufficiently small one has $h(\varphi_X(p,t))>0$ and $h(\varphi_Y(p,t))<0$.
\begin{definition} 

If $p \in \Sigma^s$, define 
        \begin{equation*}
        \varphi_Z^X(p,t)=\left\{
            \begin{array}{cc}
               \varphi_X(p,t) , & t\le 0,  \\
                \varphi_{\widetilde{Z}}(p,t), & t\ge 0 ,
            \end{array}
            \right.
\qquad \text{and} \qquad
            \varphi_Z^Y(p,t)=\left\{
            \begin{array}{cc}
               \varphi_Y(p,t) , & t\le 0,  \\
                \varphi_{\widetilde{Z}}(p,t), & t\ge 0 .
            \end{array}
            \right.
        \end{equation*}
         If $p \in \Sigma^e$, define 
        \begin{equation*}
            \varphi_Z^X(p,t)=\left\{
            \begin{array}{cc}
               \varphi_X(p,t) , & t\ge 0,  \\
                \varphi_{\widetilde{Z}}(p,t), & t\le 0 ,
            \end{array}
            \right.
  \qquad \text{and} \qquad
            \varphi_Z^Y(p,t)=\left\{
            \begin{array}{cc}
               \varphi_Y(p,t) , & t\ge 0,  \\
                \varphi_{\widetilde{Z}}(p,t), & t\le 0 .
            \end{array}
            \right.
        \end{equation*}
\end{definition}

Therefore, for $p \in \Sigma^e \cup \Sigma^s$, the set $\{\varphi_Z^X$, $\varphi_Z^Y$, $\varphi_{\widetilde{Z}}\}$ is a solution of $\dot{x}\in \mathcal{F_Z}(x)$. Now that the behavior of the trajectories of a Filippov vector field is well comprehended, it is essential to define singularities of such fields.

\subsection{Singularities of Filippov Vector Fields}
Now we present the definitions of singular points of a Filippov system $Z=(X,Y)\in \Omega^r_h(K,\mathbb{R}^n)$.

\begin{definition}
Consider  $Z=(X,Y)\in \Omega_h^r(K,\mathbb{R}^n)$ a Filippov vector field. 
\begin{enumerate}
    \item[(a)]  A point $p\in \Sigma$ is  a \textbf{tangency point}  of $X$ (resp. $Y$) if $Xh(p)=0$ (resp. $Yh(p)=0)$. We say that a tangency point $p \in \Sigma$ is a \textbf{visible fold} of $X$ (resp. $Y$) if $X^2h(p)>0$ (resp. $Y^2h(p)<0$).
Reversing the inequalities, we call the tangency point an \textbf{invisible fold}.
    \item[(b)]  A point $p\in \Sigma$ is a \textbf{$\Sigma$-regular point} of $X$ (resp. $Y$), if $Xh(p)\ne0$ (resp. $Yh(p)\ne0$).
     \item[(c)]  A point $p \in K$ is a \textbf{real singularity of Z} if $h(p)>0$ (resp. $h(p)<0$) and $X(p)=0$ (resp. $Y(p)=0$). In the case that $p \in K$ is such that $h(p)<0$ (resp. $h(p)>0$) and $X(p)=0$ (resp. $Y(p)=0$) then $p$ is a \textbf{virtual singularity of Z}.   
    \item[(d)] A point $p \in \Sigma$ is a \textbf{pseudo-equilibrium of $Z$} if $p \in \Sigma^e\cup \Sigma^s$ and $\widetilde{Z}(p)=0$.
\end{enumerate}
\end{definition}

In \autoref{fig:dobravisinv}, we illustrate some examples of tangency points.
\begin{figure}[h!]
    \centering
    \includegraphics[width=10cm]{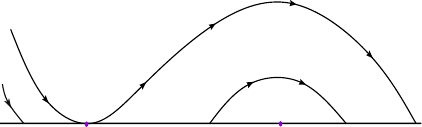}
    \put(-230,-6){\small $p_1$}
    \put(-300,0){$\Sigma$}
    \put(-310,45){$X$}
    \put(-100,-6){\small $p_2$}
    \caption{$p_1$ is a visible fold and $p_2$ an invisible fold of $X$.}
     \label{fig:dobravisinv}
\end{figure}

 The \textbf{singularities of a Filippov system} (\ref{first}) are the points $p \in Z$ such that $p$ is a point of tangency or a real singularity or a pseudo-equilibrium. 
Any other point of $Z$ is called a \textbf{regular point}. We emphasize that there exist singularities of Filippov vector fields that do not admit a stationary trajectory, meaning that they do not constitute invariant sets with respect to the flow.
\begin{definition}
When $p$ is a hyperbolic singularity of the sliding vector field $\widetilde{Z}$, it is referred to as a \textbf{hyperbolic pseudo-equilibrium}.
\end{definition}

For each $m>1$, denite by  $\mathcal{S}_{X,m}$ the set of tangency-regular points of the vector field $X$ of order $m$, i.e.
$$\mathcal{S}_{X,m}:=\{x \in \Sigma : Xh(x)=\dots =X^{m-1}h(x)=0, X^m(x)\ne0, Yh(p)\ne 0\}.$$ 
 Denote the set of visible fold regular points of $X$ by $\mathcal{S}^+_X:=\{x \in \Sigma : Xh(x)=0 \mbox{ and }X^2h(x)>0\}.$ The following holds:
\begin{proposition}\label{prop1}
The set $\mathcal{S}_X^+$ is a smooth embedded submanifold of \(\mathbb{R}^n\) of codimension \(2\).
\end{proposition}
\begin{proof}
Conisder the smooth map
\[
F:K\longrightarrow \mathbb{R}^2,\qquad 
F(p)=\bigl(h(p),\, Xh(p)\bigr).
\]
Then
\[
F^{-1}(0,0)=\{p\in K : h(p)=0,\; Xh(p)=0\}.
\]
and
\[
\mathcal{S}_X^+=\{p\in K : h(p)=0,\; Xh(p)=0,\; X^2h(p)>0\}
      =\{p\in F^{-1}(0,0):\; X^2h(p)>0\}.
\]

Let $p\in K$. For  \(v\in T_p\mathbb{R}^n\cong \mathbb{R}^n\), the differential of \(F\) at $p$ is given by
\[
dF_p(v)=\bigl(dh_p(v),\, d(Xh)_p(v)\bigr)
       =\left(\langle\nabla h(p),v\rangle,\;
              \langle\nabla(Xh)(p),v\rangle\right).
\]
Evaluating $dF_p$ at \(X_p:=X(p)\), we obtain
\[
dF_p(X_p)=\bigl(Xh(p),\, X^2h(p)\bigr).
\]
If \(p\in \mathcal{S}_X^+\), then \(Xh(p)=0\) and \(X^2h(p)>0\); hence
\[
dF_p(X_p)=(0,\;X^2h(p)),
\]
whose second component is nonzero.

On the other hand, because \(0\) is a regular value of \(h\), we have \(\nabla h(p)\neq 0\) for all
\(p\in h^{-1}(0)\). Thus, there exists \(v_0\in T_p\mathbb{R}^n\) such that
\[
dh_p(v_0)=\langle\nabla h(p),v_0\rangle\neq 0.
\]
Consequently, 
\[
dF_p(v_0)= (dh_p(v_0),\, d(Xh)_p(v_0)),
\]
has  nonzero first component.

It follows that the vectors \(dF_p(v_0)\) and  \(dF_p(X_p)\)
are linearly independent in \(\mathbb{R}^2\), and hence \(dF_p\) is
surjective for every \(p\in \mathcal{S}_X^+\). Consequently, each point of
\(\mathcal{S}_X^+\) is a regular point of the map \(F\).

By the Regular Value Theorem, the set
\(
F^{-1}(0,0)
\)
is a smooth submanifold of codimension \(2\) in a neighborhood of each
point of \(\mathcal{S}_X^+\). Since
\[
\mathcal{S}_X^+=\{p\in F^{-1}(0,0): X^2h(p)>0\}
\]
is an open subset of \(F^{-1}(0,0)\), it follows that \(\mathcal{S}_X^+\) is also
a smooth submanifold of codimension \(2\).
Thus,
\( 
\dim \mathcal{S}_X^+ = n-2.
\)
\end{proof}
We now present a more general result that is important to understand the behavior between the sliding vector field and the set of tangency-regular points on $\Sigma$. 
\begin{proposition}
     Given the Filippov vector field $Z=(X,Y)$, suppose that $h(x)=x_n$, where $x=(x_1, \dots,x_n)\in \mathbb{R}^n$. If $p \in \mathcal{S}_{X,m}$, then the sliding vector field $\widetilde{Z}$ has contact of order $m-1$ with $\mathcal{S}_{X,m}$ at $p$.
\end{proposition}
   \begin{proof}
       Denote $X=(f_1, \dots,f_n)$ where $f_i:K \rightarrow \mathbb{R}$ is a $C^r$-smooth function for each $i \in \{1, \dots, n\}$.   
       For $x \in \Sigma^{e,s}$, the sliding vector field $\widetilde{Z}(x)$ is given by:
       $$\widetilde{Z}(x)=\frac{Yh(x)X(x)-Xh(x)Y(x)}{Yh(x)-Xh(x)}.$$
       Since $p \in \mathcal{S}_{X,m}$, then $Xh(p)=0$, $Yh(p)\ne 0$ and $p \in \Sigma$. Hence $p=(p_1, \dots,  p_{n-1},0)$ and we can extend $\widetilde{Z}$ to $p$. $$\widetilde{Z}(p)=\frac{Yh(p)X(p)}{Yh(p)}=(f_1(p),\dots,f_{n-1}(p),0).$$
       Then at $p$ the sliding vector field can be extended as the restriction of the vector field $X$.
       On the other hand, the Lie derivatives have the form:
       \begin{equation*}
           \begin{array}{c}
             Xh(x)=\langle \nabla h(x),X(x)\rangle=\langle (0,\dots,0,1),(f_1(x),\dots,f_{n-1}(x),f_n(x))\rangle=f_n(x),  \\
              X^2h(x)=\langle\nabla Xh(x),X(x)\rangle= \langle (f_n^1(x),\dots ,f_n^{n-1}(x),0),(f_1(x),\dots,f_{n-1}(x),f_n(x))\rangle= \sum_{i=1}^{n-1}f_i(x)f_n^i(x), \\
              X^3h(x)=\langle\nabla X^2h(x),X(x)\rangle= \sum_{j=1}^{n-1}\sum _{i=1}^{n-1}f_j(x)(f_i(x)f_n^i(x))^j,\\
              \dots
           \end{array}
       \end{equation*}
    Suppose $m=2$, then $p$ is a fold regular point. The contact between $\widetilde{Z}$ and $\mathcal{S}_{X,m}$ at $p$ is given by: 
    $$\langle \widetilde{Z}(p),\nabla Xh(p)\rangle=\langle (f_1(p),\dots,f_{n-1}(p),0),(f_n^1(x),\dots ,f_n^{n-1}(x),0)\rangle=X^2h(p)\ne 0.$$
    Then $\widetilde{Z}$ is transversal to $\mathcal{S}_{X,m}$ at $p$. 
    The argument is analogous for arbitrary $m>2$, since $$\langle \widetilde{Z}(p),\nabla X^{m-1}h(p)\rangle=X^mh(p)\ne 0.$$
   \end{proof}

\subsection{Conley Index for Filippov Flows}\label{conleyindex}

In order to study the qualitative dynamics of Filippov systems, we make use of the Conley index theory for multivalued dynamical systems introduced by Marian Mrozek in \cite{mrozek1990cohomological}.

 Let $X$ be a locally compact metric space and consider $\pi$ a multi-valued flow defined on $X$. 
\begin{definition}
    A set $S \subset X$ is called \textbf{invariant} under the multi-valued flow $\pi$ if for each $x \in S$ there exists an orbit $\varphi:\mathbb{R}\rightarrow S$ such that $\varphi \in \mathcal{S}_{\mathbb{R}}(x)$.
\end{definition}

Given a set $U$ we define the following invariant subsets of $U$ under the multi-valued flow $\pi$ as:
\begin{equation*}
    \begin{array}{lcl}
       Inv(U)&:=  & \{x \in U : \exists \mbox{ orbit }\varphi:\mathbb{R}\rightarrow U\mbox{ such that }\varphi \in \mathcal{S}_{\mathbb{R}}(x)\}, \\
       Inv^+(U)&:=  & \{x \in U : \exists \mbox{ orbit }\varphi:\mathbb{R}^+\rightarrow U\mbox{ such that }\varphi \in \mathcal{S}_{\mathbb{R}^+}(x)\}, \\
       Inv^-(U)&:= & \{x \in U : \exists \mbox{ orbit }\varphi:\mathbb{R}^-\rightarrow U\mbox{ such that }\varphi \in \mathcal{S}_{\mathbb{R}^-}(x)\}.
    \end{array}
\end{equation*}

\begin{definition}
    A compact set $N\subset X$ is called an \textbf{isolating neighborhood} if its maximal invariant set lies in its interior, that is, $Inv(N)\cap \partial N=\emptyset$.
\end{definition}
Now, we define an isolated invariant set as usual:
\begin{definition}
    A set $S\subset X$ is called an \textbf{isolated invariant set} if there is an isolating neighborhood $N$ such that $S=Inv(N)$.
\end{definition}

\begin{definition}
    The set $U \subseteq X$ is called \textbf{strongly (positively) invariant} with respect to $\pi$ if, and only if, for each $x \in U$, $\pi(x,t)\subseteq U \mbox{ }\forall t \in \mathbb{R}$ ($\pi^+(x):=\pi(x,t)\subseteq U \mbox{ } \forall t \in \mathbb{R}^+$).
\end{definition}

 Since we are interested in the study of the behavior close to an isolated invariant set, we define the \textbf{restricted multi-valued flow}. Assuming $N$ compact, take 
 \begin{equation*}
     \begin{array}{ccl}
         \pi|_N: &N\times \mathbb{R}&\twoheadrightarrow N \cap \pi(x,\mathbb{R})\\ 
         &(x,t)&\mapsto \{y \in N : Conn_N(t,x,y,\pi)\ne \emptyset\} 
     \end{array}
 \end{equation*}
Since $Conn_N(t,x,y,\pi)\subset N$, where the trajectories are considered only up to the time at which they leave $N$, even if they return at a later time. In \cite{mrozek1990cohomological} Mrozek has proved that $\pi|_N$ is a partial multi-valued flow.

\begin{definition}\label{def18}
  A pair of subsets $(P,Q)$\footnote{We may observe that $(P,Q)$ is not necessarily a topological pair of spaces.} of an isolating neighborhood $N$ is said to be an \textbf{index pair in $N$} with respect to $\pi$ if:
    \begin{itemize}
        \item $P$ and $Q$ are compact and strongly positively invariant with respect to $\pi|_N$;
        \item $Inv^-(N)\subset int(P,N)$;
        \item $Inv^+(N)\subset N \setminus Q$;
        \item $\overline{P\setminus Q}\subset int(N)$.
    \end{itemize}
\end{definition}
\begin{remark}
    The index pair defined above is also an index pair in the Conley's sense.
\end{remark}

We now introduce the definition of the Conley index of an isolated invariant set for admissible multi-valued flows.
\begin{definition}
    If $S$ is an isolated invariant set with isolating neighborhood $N$, then the \textbf{homological Conley index} of $S$ is \begin{equation}
        \mathcal{C}(S):=H_*(P,P\cap Q),
    \end{equation} where $(P,Q)$ is an index pair in $N$ with respect to the multi-valued flow as \autoref{def18}.
\end{definition}
 Furthermore, the following equivalence holds \cite{mrozek1990cohomological}:
 $$H_*(P,Q) \cong H_*(Inv^-(N),Inv^-(N) \cap Q).$$

\begin{remark}
Mrozek emphasizes that, in the single-valued case, every index pair as given in \ref{def18} is also an index pair in the sense of Conley, see \cite{mrozek1990cohomological} Corollary 4.2.
\end{remark}

\section{Conley Index of Filippov Limit Cycles}\label{sec:CILC}

In Filippov systems, limit cycles can be classified according to their interaction with the discontinuity surface $\Sigma$. In this paper, we focus on the study of the following types:
\begin{enumerate}
    \item \textbf{Smooth Limit Cycles}\\ 
    The orbit remains entirely in a smooth region, either $X^+$ or $X^-$, and does not intersect the discontinuity surface $\Sigma$. 
    Stability and Morse index are defined as in classical smooth systems.

    \item \textbf{Crossing Limit Cycles} \\
    These are isolated periodic orbits (in the set of periodic orbits) that intersects the switching surface $\Sigma$ only at crossing points. These cycles are the closest analogues to classical smooth limit cycles. 
    
    \item \textbf{Sliding Limit Cycles} \\
    Limit cycles of the sliding vector field $\widetilde{Z}$, that is, cycles contained in the sliding or escaping regions of $\Sigma$.

    \item \textbf{Tangency Limit Cycles}\\
    The orbit tangentially touches $\Sigma$ at one or more points and these contacts are visible. In this case, the orbit does not need to be isolated within the set of periodic orbits.

    \item \textbf{Hybrid Limit Cycles} \\
    The trajectory combines different behaviors: crossing, sliding, and tangential points.

\end{enumerate}

\subsection{Crossing Limit Cycles in Dimension n}\label{subsec:crossing}

    Observe that when we consider only crossing limit cycles, the number of intersections of this trajectory and the discontinuity surface is even. For simplicity, we consider the case where the number of intersections is equal to $2$, see \autoref{fig:crossing_example}. The other cases follow with the same arguments.

\begin{figure}[h!]
    \centering
    \includegraphics[width=10cm]{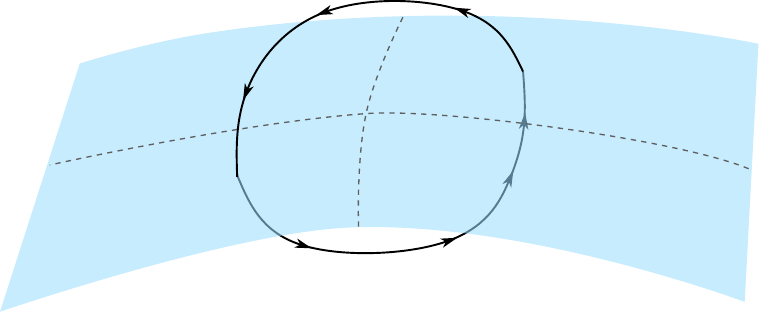}
    \put(-87,87){\small $p$}
    \put(-200,45){\small $q$}
    \put(-30,90){\small $\Sigma^c$}
    \put(-250,15){\small $\Sigma^c$}
    \put(-30,15){\small $\Sigma^{e,s}$}
    \put(-250,90){\small $\Sigma^{e,s}$}
    \caption{A crossing limit cycle in dimension 3.}
    \label{fig:crossing_example}
\end{figure}

Since the intersections of a crossing limit cycle $\gamma$ with the switching manifold $\Sigma$ are transversal, the local dynamics near $\gamma$ are well defined on both sides of $\Sigma$. Consequently, it follows that an isolating neighborhood $N$ of $\gamma$ can be constructed such that the intersection $N \cap \Sigma$ still consists solely of crossing points. In other words, the transversality condition guarantees that no sliding regions appear inside $N$, and the qualitative behavior of the trajectories in the neighborhood of $\gamma$ remains unchanged.

\subsubsection*{Construction of the first return map}
  Let $\gamma$ be a crossing limit cycle for a Filippov system which intersects the discontinuity surface at the points $p,q \in \Sigma^c$. For $r>0$ sufficiently small, take $ B_r(p) \subset \Sigma^c$ a neighborhood of $p$ in $\Sigma^c$ and consider that $Xh(p)>0, Yh(p)>0$ (the other case is similar, just reversing the inequalities). As $r > 0$ is arbitrary, we may choose such a neighborhood in which every point of $B_r(p)$ is transversal and constitutes a crossing point of $\Sigma$. Furthermore, for each $x \in B_r(p)$ it holds that $Xh(x) > 0$ and $Yh(x) > 0$. 

  \begin{figure}[h!]
    \centering
    \includegraphics[width=11cm]{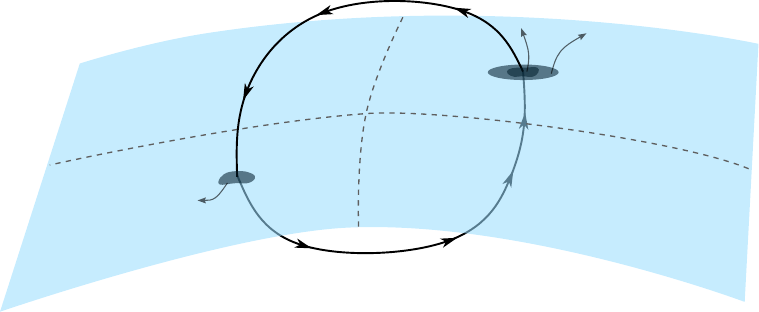}
    \put(-96,89){\scriptsize $p$}
    \put(-219,49){\scriptsize $q$}
    \put(-20,100){\scriptsize $\Sigma^c$}
    \put(-280,25){\scriptsize $\Sigma^c$}
    \put(-30,25){\scriptsize $\Sigma^{e,s}$}
    \put(-270,100){\scriptsize $\Sigma^{e,s}$}
    \put(-70,115){\small $U$}
    \put(-103,119){\small $W$}
    \put(-240,43){\small $V$}
    \caption{Construction of the first return map.}
    \label{firstreturn}
\end{figure}

  Now, since $\gamma$ is a close orbit, there exists $t_0>0$ such that $\varphi_X(t_0,p)=q\in \Sigma^c$. So, define the following function 
  \begin{equation}
      \begin{array}{ccl}
         f:\mathbb{R} \times B_r(p)  &\rightarrow &\mathbb{R}  \\
          (t,x)& \mapsto & h(\varphi_X(t,x))
      \end{array}
  \end{equation}
and observe that $f(t_0,p)=h(\varphi_X(t_0,p))=h(q)=0$ since $q \in \Sigma$. Furthermore, $\frac{\partial}{\partial t}f(t,p)=\nabla h(\varphi_X(t_0,p))\cdot X(q)=\nabla h(q)\cdot X(q)\ne 0$ since $q$ is a crossing point. Then, by the implicit function theorem, there exists a function $\tau:U\subset B_r(p)\rightarrow \mathbb{R}$ such that $\tau (p)=t_0$ and $f(\tau (x), x)= 0$ for all $x \in U$. So for all $x \in U$ the flow of the vector field $X$ returns and is transversal to the discontinuity surface $\Sigma$. 

Proceeding in the same way, consider the open neighborhood $\varphi_X(\tau(U),U)\ni q$. There exists $t_1>0$ such that $\varphi_Y(t_1,q)=p \in \Sigma^c$. We define the function:
  \begin{equation}
      \begin{array}{ccl}
         \tilde{f}:\mathbb{R} \times \varphi_X(\tau(U),U)  &\rightarrow &\mathbb{R}  \\
          (t,x)& \mapsto & h(\varphi_Y(t,x))
      \end{array}
  \end{equation}
and, again, we have $\tilde{f}(t_1,q)=h(\varphi_Y(t_1,q))=h(p)=0$ and $\frac{\partial}{\partial t}\tilde{f}(t,q)=\nabla h(p)\cdot Y(p)\ne 0$. Then there exists a function  $\tilde{\tau}:V\subset\varphi_X(\tau(U),U)\rightarrow \mathbb{R}$ such that $\tilde{\tau} (q)=t_1$ and $\tilde{f}(\tilde{\tau} (x), x)= 0$ for all $x \in V$.
Denote $W:=\varphi_Y(\tilde{\tau}(V),V)\subset \Sigma$. Therefore, it is possible to define a first return map $P:B_r(p)\rightarrow W$ by concatenating the flows $P(x)=\varphi_Y(\tilde{\tau}(\varphi_X(\tau(x),x)),\varphi_X(\tau(x),x)).$ See \autoref{firstreturn}

\begin{definition}[First return map]
    Let $\gamma$ be a crossing limit cycle intersecting $\Sigma^c$ at $p \in \gamma$. Consider a local transversal section $\ell$ through $p$, with $\ell \subset \Sigma^c$.
    The \emph{first return map} (or Poincaré map) associated to $\gamma$ is the map
    \[
        P : U \subset \ell \to \ell, \quad P(x) = \varphi_Z(T(x), x),
    \]
    where $U$ is a neighborhood of $p$ in $\ell$, $\varphi_Z(t,x)$ is the Filippov flow (the concatenation of the flows of both vector fields $X$ and $Y$, 
    and $T(x)$ is the first return time of the trajectory starting at $x \in U$ to the section $\ell$.
    By construction, $P(p) = p$.
\end{definition}
By restricting our analysis exclusively to crossing regions, the first return map becomes well defined and single-valued, in contrast to the multivalued maps that may arise in the general nonsmooth setting. As in the classical theory, the stability of the periodic orbit $\gamma$ is characterized by the derivative of the Poincaré map $P$:

\begin{definition}[Hyperbolic periodic orbit]
    Let $\gamma$ be a crossing periodic orbit of a Filippov system with associated return map $P$. 
    We say that $\gamma$ is a \emph{hyperbolic periodic orbit} if the derivative of $P$ at $p \in \gamma$ satisfies
    \[
        |P'(p)| \neq 1.
    \]
    In this case, we further classify:
    \begin{itemize}
        \item If $|P'(p)| < 1$, then $\gamma$ is an \emph{stable} limit cycle;
        \item If $|P'(p)| > 1$, then $\gamma$ is a \emph{unstable} limit cycle.
    \end{itemize}
\end{definition}
\begin{definition}[Orientable Crossing Periodic Orbit]
Let $\gamma$ be a crossing periodic orbit of a Filippov vector field, and let $P : \ell \to \ell$ be the first return map as defined above. Then $\gamma$ is said to be \emph{orientable} if 
$$\det(DP(p)) > 0.$$
\end{definition}
\begin{definition}[Morse Index of a Crossing Periodic Orbit in a Filippov System]
Let $\gamma$ be a crossing periodic orbit of a Filippov vector field, and let $P : \ell \to \ell$ the first return map associated with $\gamma$.
The \emph{Morse index} of $\gamma$ is defined as
$$\mathrm{ind}_M(\gamma) := \dim E^u,$$
where $E^u$ is the unstable subspace of the derivative $DP(p)$ evaluated at a point $p \in \gamma \cap \ell$, i.e.~the span of the eigenvectors 
corresponding to eigenvalues $\lambda$ with $|\lambda| > 1$.
\end{definition}

As in the classical setting, the following result holds.

\begin{proposition}
    Given a crossing limit cycle $\gamma$ of the Filippov system \eqref{first}. Assume that $\gamma$ is orientable and hyperbolic with Morse index $\mathrm{ind}_M(\gamma)=k$. Then $$\mathcal{C}(\gamma)\cong \widetilde{H}_*( S^k\vee S^{k+1}).$$
\end{proposition}
\begin{proof}

Since $\gamma$ is a crossing cycle, the Filippov dynamics coincides locally with a smooth vector field as discussed above. Because $\gamma$ is hyperbolic, the invariant set $S=\{\gamma\}$ is isolated. Thus there exists a compact isolating neighborhood $N$ of $\gamma$ such that $\mathrm{Inv}(N)=\gamma.$

By hyperbolicity and orientability, $N$ can be chosen homeomorphic to
$$
N \simeq D^{k} \times S^1 \times D^{n-k-1}.
$$

For a hyperbolic periodic orbit of Morse index $k$, the set $Inv^-(N)$ coincides with the local unstable manifold of $\gamma$ inside $N$, hence
$$\mathrm{Inv}^-(N) \simeq D^{k} \times S^1.$$

Let $P$ be a compact neighborhood of $\mathrm{Inv}^-(N)$ such that
$$ \mathrm{Inv}^-(N) \subset \mathrm{int}(P) \subset P \subset N,$$
and define $$Q := P \cap (\partial N\times I_\epsilon),$$ where $I_\epsilon=[-\epsilon,\epsilon]$ and $\epsilon>0$ sufficiently small.
We verify that $(P,Q)$ is an index pair.

\textit{(i) Compactness.}
$P$ is compact by construction and $Q$ is a closed subset of the compact set $P$, hence compact.

\textit{(ii) Strong positive invariance.}
We prove that $P$ is strongly positively invariant in $N$.

Let $x\in P$ and suppose that a trajectory $\varphi([0,t],x)$ satisfies
\[
\varphi([0,t],x)\subset N.
\]
We must show that $\varphi([0,t],x)\subset P$.

Assume by contradiction that there exists a first time
\[
t_0=\inf\{s>0:\varphi(s,x)\notin P\}.
\]
Then
\[
\varphi([0,t_0),x)\subset P
\quad\text{and}\quad
\varphi(t_0,x)\in \partial P.
\]

Since $P\subset N$, we have $\varphi(t_0,x)\in N$.  
Near $\gamma$ the flow admits the hyperbolic product structure
\[
N \simeq D^k \times S^1 \times D^{n-k-1},
\]
where $D^k$ corresponds to unstable directions and $D^{n-k-1}$ to stable directions.  
Moreover,
\[
\mathrm{Inv}^-(N)\simeq D^k \times S^1
\]
is positively invariant while trajectories remain in $N$.

The neighborhood $P$ is chosen as a sufficiently small compact of 
$\mathrm{Inv}^-(N)$ inside $N$, that is, a product neighborhood in the stable directions.  
Along the boundary components corresponding to the stable directions, the vector field is inward pointing for positive time. Hence trajectories cannot exit $P$ through those boundary components.

Therefore the only possible way to leave $P$ while remaining in $N$ would be through the portion of $\partial P$ contained in $\partial N$, namely $Q=P\cap \partial N$. But this contradicts the assumption that $\varphi([0,t],x)\subset N$ and has not reached $\partial N$ before $t_0$.

Hence $\varphi([0,t],x)\subset P$, and $P$ is strongly positively invariant in $N$. The argument to prove that $Q$ is also strong positive invariant is the same.

\textit{(iii) $\mathrm{Inv}^-(N)\subset \operatorname{int}(P,N)$.}
This holds by construction.

\textit{(iv) $\mathrm{Inv}^+(N)\subset N\setminus Q$.}
Every point of $Q$ lies close enough to $\partial N$ and is an exit point of the isolating block. Hence no trajectory starting at $Q$ remains in $N$ for positive time, so $Q\cap \mathrm{Inv}^+(N)=\emptyset$.

\textit{(v) Isolation property.}
Since $\mathrm{Inv}(N)=\gamma$ and $\mathrm{Inv}^-(N)\subset \operatorname{int}(P)$, we have
\[
\gamma \subset \operatorname{int}(P\setminus Q),
\quad\text{and}\quad
\overline{P\setminus Q}\subset \operatorname{int}(N).
\]

Therefore $(P,Q)$ is an index pair for $\gamma$.

By definition, the Conley index is the homotopy type 
$$
\mathcal{C}(\gamma) := H_*(P,Q).
$$

Since $P$ is a compact neighborhood of $\mathrm{Inv}^-(N) \simeq D^{k} \times S^1$ inside $N \simeq D^{k} \times S^1 \times D^{n-k-1}$, and $Q \simeq P \cap \partial N$ it is a good pair, then $H_*(P,Q) \cong \widetilde{H}_*(P/Q)$. 

Finally, the quotient space is homotopy equivalent to
$$P/Q \simeq  \frac{D^k \times S^1}{\partial D^k \times S^1}.$$
Using standard homotopy equivalences, we have:
\[
\frac{D^k \times S^1}{\partial D^k \times S^1} \simeq S^k \vee S^{k+1}.
\]

Hence,
\[
\mathcal{C}(\gamma) \cong \widetilde{H}_*(S^k \vee S^{k+1}),
\]
as claimed.

\end{proof}

\subsection{Sliding Limit Cycles in Dimension n}\label{subsec:sliding}
In this case, we are considering only limit cycles $\gamma$ in the sliding vector field $\widetilde{Z}$. Observe that, since this orbit lives in the sliding/escaping region, we have the property that $Xh(p)Yh(p)<0$ for all $p \in \gamma$. Thus, the behavior of such an orbit as a cycle of the Filippov system can only attract or repel the regular vector fields $X$ and $Y$.
See \autoref{fig:puresliding_N}.

\begin{figure}[h!]
    \centering
    \includegraphics[width=15cm]{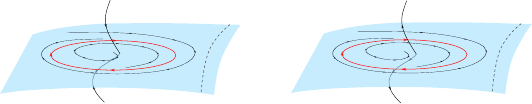}
    \put(-410,15){\small $\Sigma^s$}
    \put(-180,15){\small $\Sigma^e$}
    \put(-272,0){\small $\partial \Sigma^s$}
    \put(-38,0){\small $\partial \Sigma^e$}
    \put(-285,35){\small \textcolor{red}{$\gamma$}}
    \put(-52,35){\small \textcolor{red}{$\gamma$}}
    \caption{Limit cycle $\gamma$ contained the sliding and escaping regions, respectively.}
    \label{fig:puresliding_N}
\end{figure}

The first return map is defined then as the first return map given by the sliding vector field and we characterize $\gamma$ with respect to the behavior of the regular vector fields $X$ and $Y$.
So we have the following result:
\begin{proposition}
    Let $\gamma \in \Sigma$ be an orientable and hyperbolic limit cycle of the vector field $\widetilde{Z}$ with index $ind_M(\gamma)=k$ (with respect to $\widetilde{Z}$). Then the Conley index of $\gamma$ in the Filippov vector field is: $\mathcal{C}(\gamma)\cong \widetilde{H}_*(S^{k+2}\vee S^{k+1})$ if $\gamma \in \Sigma^e$ and $\mathcal{C}(\gamma)\cong \widetilde{H}_*(S^{k+1}\vee S^k)$ if $\gamma \in \Sigma^s$.
\end{proposition}

\begin{proof}
Let $\gamma\subset\Sigma$ be an orientable, hyperbolic limit cycle of the sliding vector field $\widetilde{Z}$ and denote by $ind_M(\gamma)=k$ its Morse index as a periodic orbit of $\widetilde{Z}$ considered on $\Sigma$ (so the classical Conley index of $\gamma$ for the sliding flow is homotopy equivalent to a wedge $S^{k}\vee S^{k+1}$).

Because $\gamma$ is hyperbolic for $\widetilde{Z}$ and lies in the sliding/escaping region we have $Xh(p)\,Yh(p)<0$ for every $p\in\gamma$. Choose a small tubular neighborhood of codimension 1 $N_\Sigma$ of $\gamma$ inside the switching manifold $\Sigma$ and a short normal interval $I$ transversal to $\Sigma$ so that
\[
N \;=\; N_\Sigma\times I
\]
is a compact neighborhood of $\gamma$. Observe that $N$ can be taken as an isolating neighborhood for the Filippov flow since the sign condition on $Xh$ and $Yh$ guarantees no other invariant set inside $N$.

We analyze the sets $Inv^\pm (N)$ according to whether $\gamma$ lies in the escaping region $\Sigma^e$ or in the sliding region $\Sigma^s$.

\medskip\noindent\textbf{(i) Escaping region $\Sigma^e$.}
In this case, both regular vector fields point away from $\Sigma$. For each $p\in N_\Sigma$, the normal direction along the interval $I$ adds one extra unstable (outgoing) direction: points on both sides of $\Sigma$ leave the neighborhood $N$ in positive time, while the trajectory lives in $N_\Sigma$ in negative time.
This implies that all points in this nromal direction dimension belongs to the set $Inv^-(N)$. 

Combining this with the classical stability of $\widetilde{Z}$ we have $$Inv^-(N) \simeq S^1 \times D^k \times I \simeq S^1 \times D^{k+1}\mbox{ and } Q\simeq Inv^-(N) \cap \partial N=S^1 \times \partial D^{k+1}, $$
which implies that $\mathcal{C}(\gamma)=H_*(P,Q)\simeq \widetilde{H}_*(P/Q)\simeq \widetilde{H}_*(S^{k+1}\vee S^{k+2}).$

\medskip\noindent\textbf{(ii) Sliding region $\Sigma^s$.}
In the sliding case, both regular vector fields point toward $\Sigma$. This means that the normal direction is contracting: nearby trajectories are attracted to $\Sigma$ instead of moving away from it.
 In this case all points in this extra dimension belongs to the set $Inv^+(N)$ since the flow enters the sliding region in positive time. Then $$Inv^-(N)\simeq S^1 \times D^k \mbox{ and } Q\simeq Inv^-(N) \cap \partial N =S^1 \times \partial D^k,$$ and $$\mathcal{C}(\gamma)=H_*(P,Q)\simeq \widetilde{H}^*(S^k \vee S^{k+1}).$$
\end{proof}

\subsection{Tangential Cycles in Dimension $2$ and $3$}\label{subsec:tangential}
In the case of tangential cycles, we are interested in tangency points of even order, since the trajectory must return to the discontinuity surface. In this section, we consider $\gamma$ to be a limit cycle of a Filippov vector field that is tangent to $\Sigma$ at a visible fold-regular point $p$, i.e., $Xh(p)=0$, $X^2h(p)>0$ and $Yh(p)\ne 0$.

In this scenario, even if $\gamma$ is an attractor for the vector field $X$, a small neighborhood of $\gamma$ can have an ``exit set'' for the Filippov flow depending on the sign of $Yh(p)$, as shown in \autoref{creat_exit}. We can see that while $\gamma$ is an attracting periodic orbit for the vector field $X$, in the Filippov context, all points in the darker shaded region (pink) belong to $Inv^-(N)$, where $N$ is a neighborhood of $\gamma$ represented by the lighter shaded region (blue).

\begin{figure}[H]
    \centering
    \includegraphics[width=12cm]{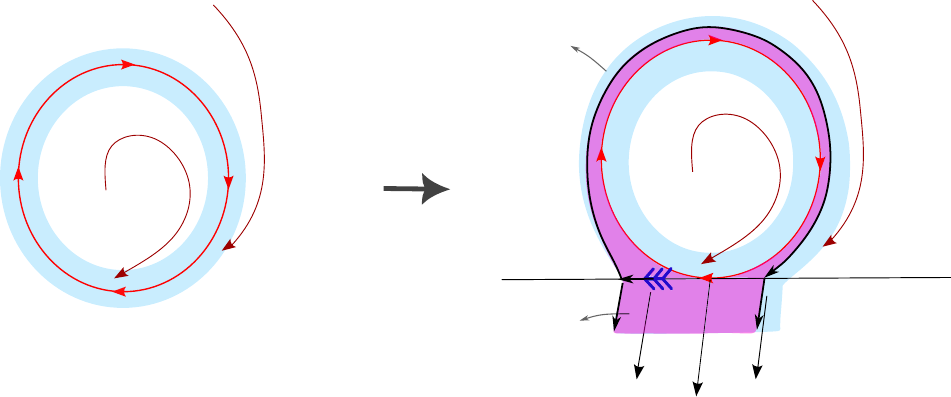}
    \put(-300,20){\small \textcolor{red}{$\gamma$}}
    \put(5,40){\small $\Sigma$}
    \put(-87,35){\small $p$}
    \put(-165,18){\small $Inv^-(N)$}
    \put(-146,125){\small $N$}
    \caption{Creation of an ``exit'' set in $N$.}
    \label{creat_exit}
\end{figure}

\begin{proposition}
Let $Z=(X,Y)\in \Omega_h^r(K,\mathbb{R}^n)$ be a Filippov vector field with $n\in\{2,3\}$. 
Let $\gamma$ be an orientable hyperbolic limit cycle of the vector field $X$, having a single visible fold-regular tangency with the switching manifold $\Sigma$ at a point $p$. Let $P$ be the first return map associated with the flow of $X$ around $\gamma$, and let $\lambda$ be the eigenvalue of $DP(p)$ whose eigenvectors generate a subspace normal to $\Sigma$ at $p$. Assume further that the sliding region near $p$ contains no pseudo-equilibria.
Then the Conley index of $\gamma$ in the Filippov system $Z$ is given by
\[
\mathcal{C}(\gamma) \cong
\begin{cases}
\widetilde{H}_*(S^k \vee S^{k+1}), & \text{if } (|\lambda|<1,\,Yh(p)>0)\ \text{or}\ (|\lambda|>1,\,Yh(p)<0),\\[1ex]
\widetilde{H}_*(*), & \text{if } (|\lambda|<1,\,Yh(p)<0)\ \text{or}\ (|\lambda|>1,\,Yh(p)>0),
\end{cases}
\]
where $k \in \{0, \dots, n-1\}$ is the Morse index of $\gamma$ with respect to the vector field $X$.
\end{proposition}

\begin{proof}
Denote by $\varphi_Z(t,x)$ the Filippov flow generated by $Z$.

\textbf{Local structure near the tangency.}

Since $p$ is a visible fold-regular point, that is, $Xh(p)=0$, $X^2h(p)>0$ and $Yh(p) \ne 0$, there exists a neighborhood $U$ of $p$ such that the switching manifold decomposes locally as
\[
\Sigma = \Sigma^c \cup \Sigma^{s/e} \cup \mathcal{S}_X^+,
\]
where
\begin{itemize}
    \item $\Sigma^c$ is the crossing region;
    \item $\Sigma^{s/e}$ is the sliding region if $Yh(p)>0$, or the escaping region if $Yh(p)<0$;
    \item $\mathcal{S}_X^+$ is the set of fold points of $X$, which is a point if $n=2$ and a smooth curve if $n=3$.
\end{itemize}
Since $Yh(p)\ne 0$ and $\widetilde{Z}$ does not admit an equilibrium point near $p$, then locally there exist no tangency points with respect to the vector field $Y$. Furthermore, $\widetilde{Z}$ is smooth and locally transversal to the set $\mathcal{S}_X^+$ in a neighborhood of $p$.

\textbf{Construction of the isolating neighborhood for the Filippov vector field.}

Since $\gamma$ is a hyperbolic periodic orbit of the vector field $X$ with Morse index $k$, there exists a tubular neighborhood of $\gamma$
$$N_0 \simeq S^1 \times D^k \times D^{n-k-1},$$
for $n=2,3$. This neighborhood is a round handle for the orbit: the factor $S^1$ corresponds to the central direction induced by the periodic orbit, while $D^k$ and $D^{\,n-k-1}$ represent, respectively, the unstable and stable transverse directions.

Because $\gamma$ intersects the switching manifold $\Sigma$ at the fold point $p$, we adapt this neighborhood to the Filippov dynamics.

Let $$N_X := N_0 \cap \{x\in K : h(x)\ge0\}.$$
Thus $N_X$ is the portion of the tubular neighborhood located on the region where the dynamics is governed by the vector field $X$. Geometrically this operation cuts the round handle along the switching manifold, so that part of $\Sigma$ becomes part of the boundary of the neighborhood $N_X$.

Observe that, in a small neighborhood of $\gamma$ in $N_X$, the orbits of the flow may either approach  $\gamma$, be repelled by $\gamma$ or even enter into $\gamma$. However, a comprehensive understanding of this behavior is only possible once we analyze the Filippov flow, as illustrated in \autoref{creat_exit}. 
Keeping this in mind, we can define $\mathcal{\widetilde{A}} \subset N_X$ in the following way.

 For $\epsilon > 0$ sufficiently small, define
\[
\xi := \mathcal{S}_X^+ \cap \overline{B}_\epsilon(p).
\]
$\xi$ is an $(n-2)$-dimensional submanifold of $\Sigma$ by \autoref{prop1}.
Consider the case $Yh(p) > 0$; the other case is analogous by reversing time.
Since $\widetilde{Z}$ is transversal to $\mathcal{S}_X^+$, for each $x \in \xi$, there exists $\tau_x > 0$ sufficiently small such that
\[
\alpha := \{ \varphi_{\widetilde Z}(-\tau_x,x) ; x \in \xi \}
\]
is an $(n-2)$-dimensional submanifold contained in the sliding region, with $n=2,3$. 
For $y =\varphi_{\widetilde Z}(-\tau_x,x) \in \alpha$ we define the curve $\beta:=\{\varphi_X(-\iota_y,y)\}$ contained in the crossing region, where $\iota_y>0$ is the smallest value such that the negative flow by the vector field $X$ touches $\Sigma$. By uniqueness of solutions of the vector field $X$, the shell type set $\mathcal{A}$ is foliated by orbits of $\varphi_X$, where $$\mathcal{A}:=\{\varphi_X(t,y); y \in \alpha \mbox{ and } -\iota_y\le t \le0\}.$$ This makes it easier to understand the behavior of the Filippov vector field. See Figures \ref{set_a} and \ref{Set_A} (A).

\begin{figure}[h]
    \centering
    \includegraphics[width=13cm]{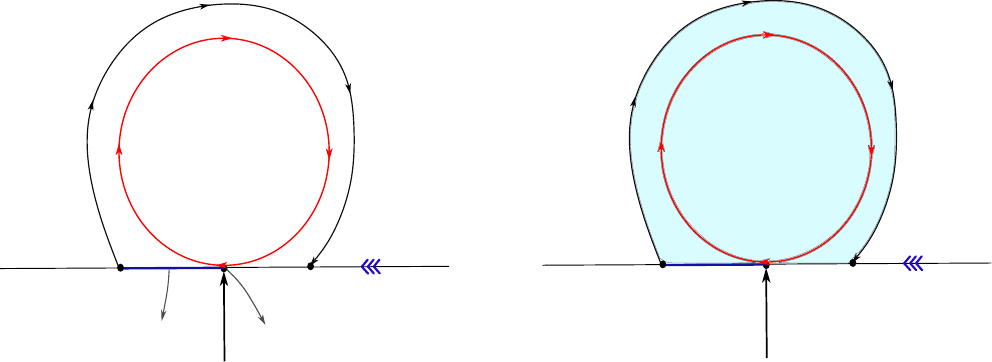}
    \put(-270,5){\small $\xi=\{p\}$}
    \put(-315,8){\small $\eta$}
    \put(-330,25){\small $\beta$}
    \put(-258,25){\small $\alpha$}
    \put(-89,75){\small $\overline{\mathcal{A}}$}
    \put(-200,33){\small $\Sigma$}
    \put(2,33){\small $\Sigma$}
    \caption{Construction of the set $\overline{\mathcal{A}}$ for $n=2$.}
    \label{set_a}
\end{figure}

\begin{figure}[h!]
    \centering
    \includegraphics[width=12cm]{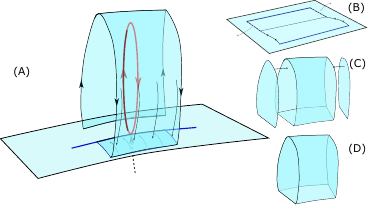}
    \put(-288,50){\small $\mathcal{S}_X^+$}
    \put(-220,56){\footnotesize $p$}
    \put(-258,60){\small $\xi$}
    \put(-248,42){\small $\alpha$}
    \put(-272,70){\small $\beta$}
    \put(-55,152){\scriptsize \textcolor{gray}{$\alpha$}}
    \put(-84,177){\scriptsize \textcolor{gray}{$\beta$}}
    \put(-72,162){\scriptsize \textcolor{gray}{$\xi$}}
     \put(-73,170){\small $D_1$}
     \put(-100,100){\small $D_2$}
     \put(-27,106){\small $D_3$}
    \put(-130,153){\small $\eta_0$}
    \put(-34,188){\small $\eta_1$}
    \put(-98,15){\footnotesize $\partial \mathcal{\overline{A}}$}
    \caption{Construction of the boundary of the set $\overline{\mathcal{A}}$ for $n=3$.}
    \label{Set_A}
\end{figure}

For $n=2$, consider the region between $\beta$ and $\xi$ in $\Sigma$. This region reduces to a curve $\eta$ joining the points $\beta$ and $\xi$. In this case, define $\mathcal{\overline{A}}$ as the region delimited by $\mathcal{A} \cup \eta$ which is topologically a 2-disk.

For $n=3$, let $\eta_0$ and $\eta_1$ be simple curves in $\Sigma$ connecting the endpoints of the curves $\beta$ and $\xi$, with $\eta_0 \cap \eta_1 = \emptyset$. The concatenation of these curves,  $\beta \cup \xi \cup \eta_0 \cup \eta_1,$ delimits a   
$2$-disk $D_1$ in $\Sigma$,  see \autoref{Set_A}(B).
 Now, we attach two $2$-disks, $D_2, D_3$, to the set $\mathcal{A}$ as illustrated in \autoref{Set_A}(C). Topologically, the resulting set is a 2-sphere \autoref{Set_A} (D).
Denote $\overline{\mathcal{A}}$ the $3$-disk delimited by the sphere like region constructed above. Finally, for both cases ($n=2,3$), define
$$\widetilde{\mathcal{A}} := \overline{\mathcal{A}} \cap N_X.$$

To incorporate the dynamics of $Y$ in $h<0$, we attach a flow box generated by the vector field $Y$ below the switching manifold.
Since we are assuming $Yh(p)>0$, define 
\[
N_Y :=
\{\varphi_Y(t,x) : x\in \widetilde{\mathcal{A}} \cap \Sigma,\ t\in[-\delta,0]\},
\]
for $\delta>0$ sufficiently small.

Topologically,
\[
N_Y \simeq D^1 \times D^k \times D^{\,n-k-1}.
\]

\begin{figure}[h!]
    \centering
    \includegraphics[width=12cm]{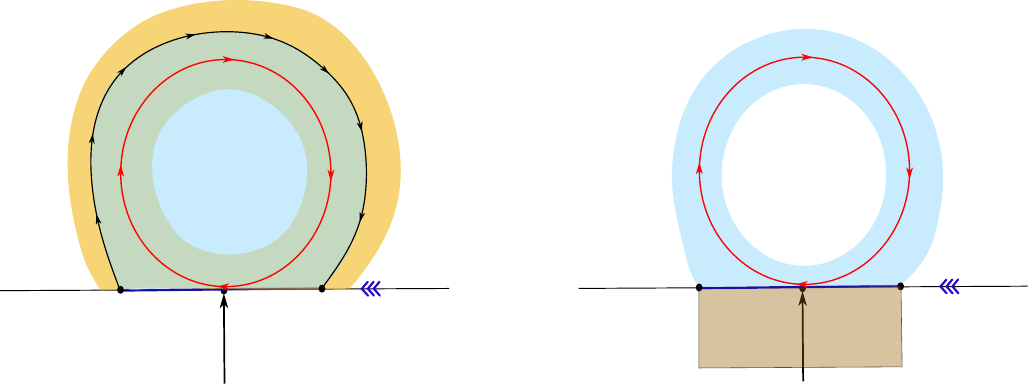}
    \put(-275,25){\small $p$}
    \put(-83,25){\small $p$}
    \put(-352,28){\small $\Sigma$}
    \put(-158,28){\small $\Sigma$}
    \put(-305,22){\small $\beta$}
    \put(-238,24){\small $\alpha$}
    \put(-228,100){\small $N_X$}
    \put(-270,70){\small $\overline{\mathcal{A}}$}
    \put(-130,70){\small $\widetilde{\mathcal{A}}$}
    \put(-125,10){\small $N_Y$}
    \caption{Isolating neighborhood $N$ for $n=2$.}
    \label{N_2}
\end{figure}

We then define an isolating neighborhood by taking the union
\[
N :=  \mathcal{\widetilde{A}} \cup N_Y,
\]
see Figures \ref{N_2} and \ref{N_3}.

\begin{figure}[h!]
    \centering
    \includegraphics[width=10cm]{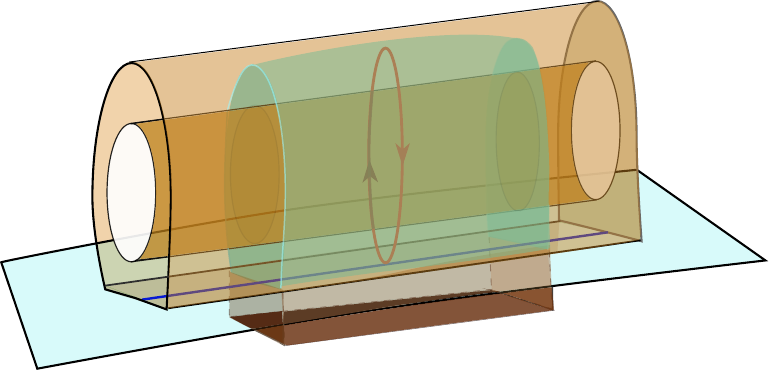}
    \put(-143,34){\scriptsize $p$}
    \put(-75,11){\small $N_Y$}
    \put(-207,100){\small $\mathcal{\overline{A}}$}
    \put(-45,97){\small $N_X$}
    \put(0,40){\small $\Sigma$}
    \put(-145,124){\small $\gamma$}
    \put(-248,20){\small $\mathcal{S}_X^+$}
    \caption{Isolating neighborhood $N$ for $n=3$.}
    \label{N_3}
\end{figure}

Observe that every Filippov trajectory starting in $N\setminus \{\gamma\}$ leaves $N$ in at least one time direction for $n \in \{2,3\}$ and for all $k \in \{0, \dots, n-1\}$. Consequently $\mathrm{Inv}(N)=\gamma,$ so $N$ is an isolating neighborhood for the Filippov flow.

\textbf{Construction of an index pair.}

Following the construction of index pairs in the sense of Mrozek we take a small compact neighborhood $P$ of the negative invariant set $\mathrm{Inv}^-(N)$
and $Q\simeq \partial P\cap \partial N.$
Then $(P,Q)$ is an index pair for $\gamma$, and the Conley index is given by
\[
\mathcal C(\gamma)\cong H_*(P,Q).
\]

The following properties hold:
\begin{enumerate}
    \item If $Yh(p)>0$ then $\Sigma^{s/e}$ is a sliding region and the additional boundary component $\partial N_Y$ contains at least one trajectory of the set $Inv^+(N)$;
    \item If $Yh(p)<0$ the region $\Sigma^{s/e}$ is an escaping region and $\partial N_Y$ contains at least one trajectory that belongs to $Inv^-(N)$. 
\end{enumerate}
\textbf{Conley index computation in dimension $n=2$:}

In this case, $\dim(\Sigma)=1$ and $\mathcal{S}_X^+=\{p\}$.

\begin{itemize}
\item For $Yh(p)>0$, we have that all trajectories starting at a point $x \in N_Y$ belongs to $Inv^+(N).$ Indeed, for all $x\in N_Y$ there exists a smaller value $t_x>0$ such that $\varphi_Y(t_x,x) \in \Sigma$ and there exist three possibilities:
    \begin{enumerate}
        \item $\varphi_Y(t_x,x) =p$, or
        \item $\varphi_Y(t_x,x) \in \Sigma^s$, or
        \item $\varphi_Y(t_x,x) \in \Sigma^c$.
    \end{enumerate}
    If $\varphi_Y(t_x,x) =p$, since $p \in \gamma$, the Filippov trajectory $\varphi_Z(t,x)$ goes into $\gamma$ for positive time $t$. Hence $x \in Inv^+(N)$.

    If $\varphi_Y(t_x,x)=y \in \Sigma^s$, then $\varphi_{\widetilde{Z}}(t,y) \rightarrow p$. This follows since the vector field $\widetilde{Z}$ is one dimensional and there are no pseudo-equilibrium points in this region (by assumption). Hence $x \in Inv^+(N)$.
    
    If $\varphi_Y(t_x,x) \in \Sigma^c$ then the Filippov trajectory goes through the crossing region of the switching manifold and the dynamics is now controlled by the vector field $X$. Since $\gamma$ is a limit cycle, there exists $\widetilde t_x>0$ such that $\varphi_X(\widetilde t_x,\varphi_Y(t_x,x)) \in \Sigma^s$. Using the argument above, once this orbit reaches the sliding region, its future is $p$. Hence, again the orbit goes into $\gamma$ and thus, $x \in Inv^+(N)$.

These arguments prove that every point in $N$ outside the orbit $\gamma$ belongs to $Inv^+(N)$.
Studying the different stability behaviors for $\gamma$ we have:

\begin{figure}[h!]
    \centering
    \includegraphics[width=12cm]{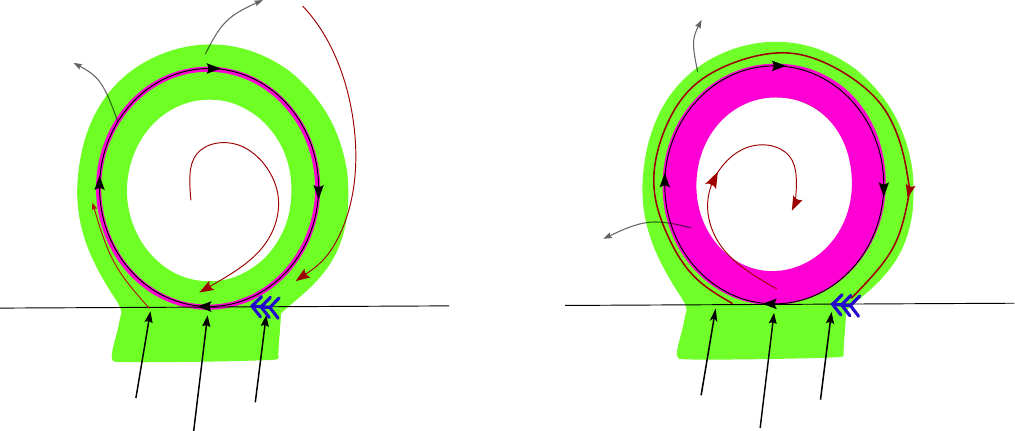}
    \put(-78,33){\small $p$}
    \put(-270,33){\small $p$}
    \put(-250,145){\footnotesize $Inv^+(N)$}
    \put(-108,145){\footnotesize $Inv^+(N)$}
    \put(-168,55){\footnotesize $Inv^-(N)$}
    \put(-240,35){\footnotesize $\Sigma^s$}
    \put(-48,35){\footnotesize $\Sigma^s$}
    \put(-315,35){\footnotesize $\Sigma^c$}
    \put(-123,35){\footnotesize $\Sigma^c$}
    \put(-393,128){\footnotesize $\gamma=Inv^+(N)\cap Inv^-(N)$}
    \caption{For $Yh(p)>0$, $|\lambda|<1$ and $|\lambda|>1$ respectively.}
    \label{fig:case1_att}
\end{figure}

    \begin{itemize}
    \item Case $|\lambda|<1$ ($k=0$, \autoref{fig:case1_att}): the limit cycle is attracting both inside of $\gamma$ under the flow of $X$ as well as for the flows of $Y$ and $\widetilde{Z}$ outside $\gamma$. Therefore, $Inv^-(N)=\gamma$ and $Inv^-(N)\cap \partial N=\emptyset$. Then 
    $$\mathcal{C}(\gamma)=H_*(S^1)\simeq \widetilde H_*(S^0 \vee S^1).$$

    \item Case $|\lambda|>1$ ($k=1$,\autoref{fig:case1_att} ): the limit cycle is repelling in the region in $N$ inside $\gamma$ and thus, this region belongs to $Inv^-(N)$. On the other hand, the region in $N$ outside $\gamma$ is attracting for both flows $\widetilde{Z}$ and $Y$, therefore belongs to $Inv^+(N)$. Hence $Inv^-(N)\simeq S^1 \times D^1$ and $Q\simeq S^1$. Thus 
    $$\mathcal{C}(\gamma)=H_*(S^1\times D^1,S^1)\cong \widetilde H_*(D^2).$$
\end{itemize}

\item For $Yh(p)<0$ all trajectories starting at a point $x \in N_Y$ belongs to $Inv^-(N).$ The argument is similar to the previous case by reversing time. 

\begin{figure}[h!]
    \centering
    \includegraphics[width=12cm]{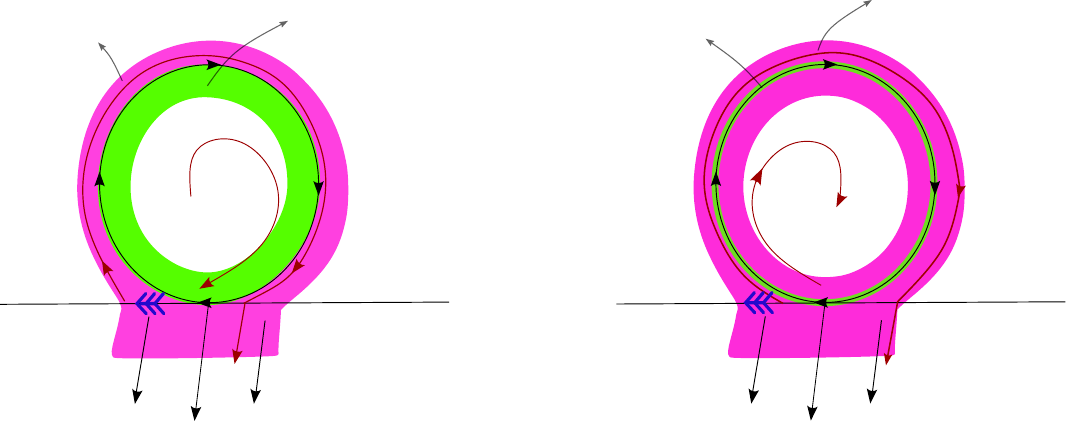}
    \put(-74,30){\small $p$}
    \put(-272,30){\small $p$}
    \put(-247,127){\footnotesize $Inv^+(N)$}
    \put(-330,125){\footnotesize $Inv^-(N)$}
    \put(-60,137){\footnotesize $Inv^-(N)$}
    \put(-244,30){\footnotesize $\Sigma^c$}
    \put(-48,30){\footnotesize $\Sigma^c$}
    \put(-315,30){\footnotesize $\Sigma^e$}
    \put(-123,30){\footnotesize $\Sigma^e$}
    \put(-178,128){\footnotesize $\gamma=Inv^+(N)\cap Inv^-(N)$}
    \caption{For $Yh(p)<0$, $|\lambda|<1$ and $|\lambda|>1$ respectively.}
    \label{fig:case2_rep}
\end{figure}
    \begin{itemize}
    \item Case $|\lambda|<1$ ($k=0$, \autoref{fig:case2_rep}): then the limit cycle is attracting in $N$ inside $\gamma$ and is repelling outside $\gamma$. In this case $Inv^-(N)\simeq S^1 \times D^1$ and $Q\simeq S^1$ which implies $$\mathcal{C}(\gamma)=H_*(S^1\times D^1,S^1)\cong \widetilde H_*(D^2)\cong 0.$$

    \item Case $|\lambda|>1$ ($k=1$, \autoref{fig:case2_rep}): then the limit cycle is repelling for both sides of $\gamma$, hence $N$ belongs to $Inv^-(N)$. Then $Inv^-(N)=S^1 \times D^1$, $Q\simeq S^1 \times \partial D^1$ and
    $$\mathcal{C}(\gamma)=H_*(S^1 \times D^1,S^1 \times \partial D^1)\cong \widetilde{H}_*(S^1 \vee S^2).$$
\end{itemize}
\end{itemize}

\textbf{Conley index in dimension $n=3$}

Now assume that $n=3$, then $\mathrm{dim} (\Sigma)=2$ and $\mathrm{dim}(\mathcal{S}_X^+)=1$. We begin with a study of the behavior of the Filippov multi-valued flow $\varphi_Z$ around $\gamma$. Consider $Yh(p)>0$, then in a neighborhood of $p$ one has $\Sigma=\Sigma^c \cup \Sigma^s \cup \mathcal{S}_X^+$. The flow on the sliding vector field can either be attracted or repelled to the point $p$ depending on the stability of $\gamma$. Observe the Figures \ref{dim3_atrat_p}, \ref{dim3_repel_p}:

\begin{figure}[H]
    \centering
    \includegraphics[width=14cm]{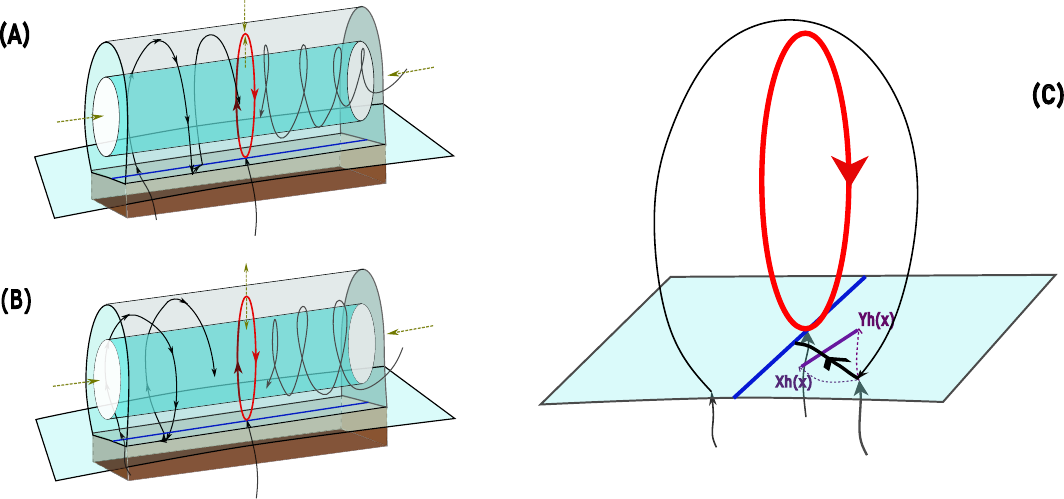}
    \put(-74,41){\footnotesize $x$}
    \put(-30,75){\footnotesize $\Sigma^s$}
    \put(-180,38){\footnotesize $\Sigma^c$}
    \put(-130,29){\footnotesize $\mathcal{S}_X^+$}
    \caption{ (A): $Yh(p)>0, |\lambda |<1, k=0$, (B): $Yh(p)>0, |\lambda|>1, k=1$, (C): behavior in $\widetilde{Z}$.}
    \label{dim3_atrat_p}
\end{figure}

\begin{figure}[H]
            \centering
            \includegraphics[width=14cm]{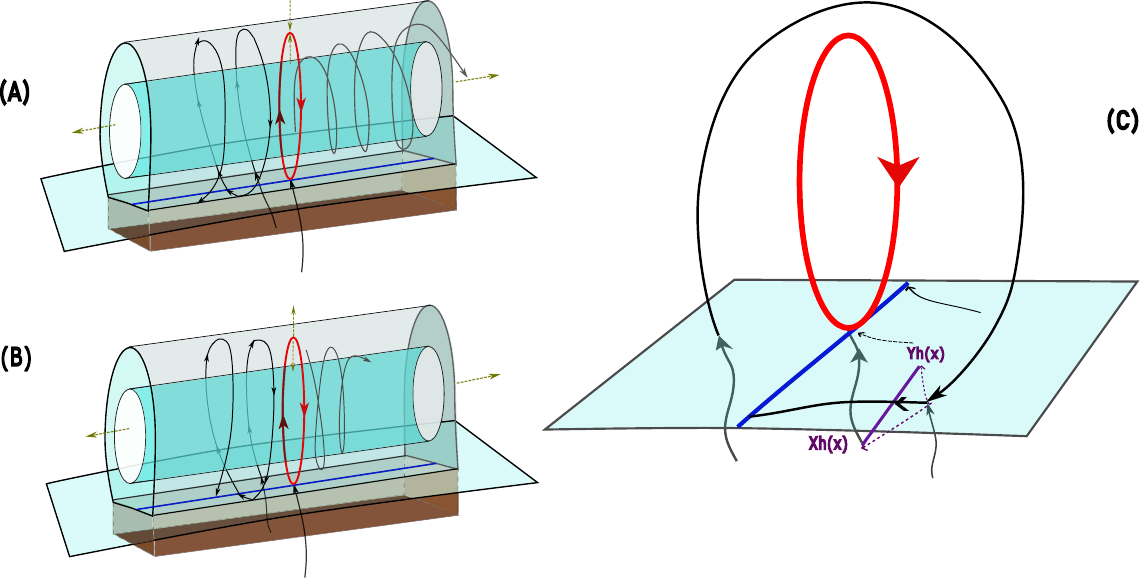}
            \put(-70,58){\footnotesize $x$}
            \put(-30,90){\footnotesize $\Sigma^s$}
            \put(-183,58){\footnotesize $\Sigma^c$}
            \put(-146,44){\footnotesize $\mathcal{S}_X^+$}
            \caption{(A): $Yh(p)>0, |\lambda |<1, k=1$, (B): $Yh(p)>0, |\lambda|>1, k=2$, (C): behavior in $\widetilde{Z}$.}
            \label{dim3_repel_p}
 \end{figure}

The (olive green) dashed arrows are not part of the flow, they merely represent the attracting or repelling nature of the flow.

 In \autoref{dim3_atrat_p} it can be observed that the flow for the sliding vector field exhibits an attracting behavior towards the point $p$ since the trajectories of the vector field $X$ tends to the periodic orbit in the same direction of $\mathcal{S}_X$. That is, it doesn't matter if $|\lambda|\lessgtr 1$ as the behavior on the sliding vector field is almost the same, represented in \textbf{(C)}. 
 
 The discussion is analogous in \autoref{dim3_repel_p}, but now we may observe that there exists at least one trajectory of the sliding vector field that converges to the point $p$ while the others have a repelling behavior for each side of $p$. This implies that, in both cases, the set $Inv^+(N)$ intersects the sliding vector field near $p$. 
 
For $Yh(p)<0$ we have an escaping region instead of a sliding region. The sliding flow behaves similar but reversing time and the set $Inv^-(N)$ intersects the sliding vector field near $p$. See Figures \ref{dim3_atrat_n} and \ref{dim3_repel_n}.

    \begin{figure}[h!]
     \centering
     \includegraphics[width=14cm]{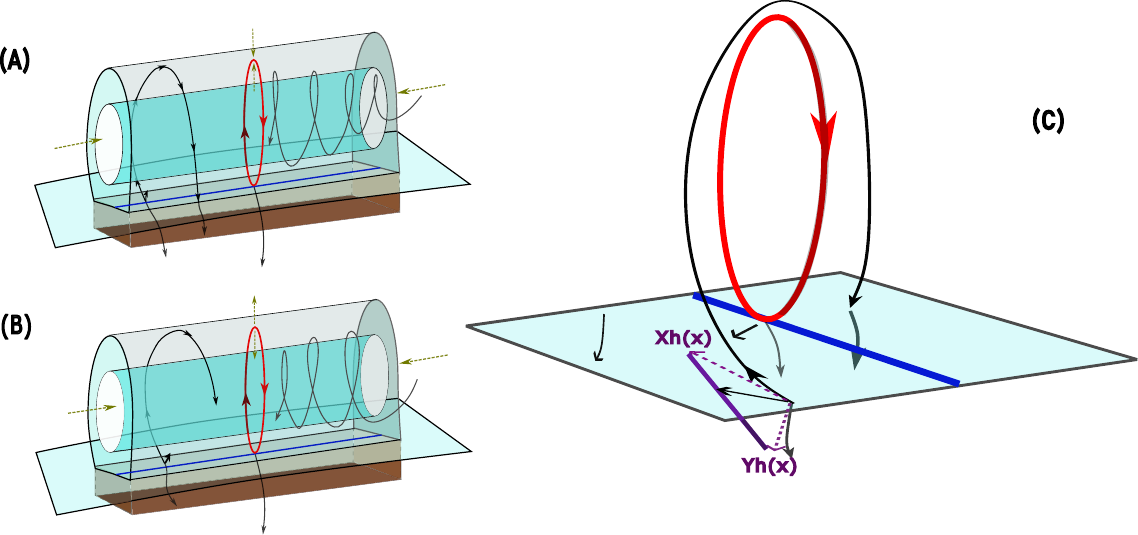}
     \put(-120,46){\footnotesize $x$}
    \put(-40,60){\footnotesize $\Sigma^c$}
    \put(-210,68){\footnotesize $\Sigma^e$}
    \put(-64,44){\footnotesize $\mathcal{S}_X^+$}
     \caption{(A): $Yh(p)<0, |\lambda |<1, k=0$, (B): $Yh(p)<0, |\lambda|>1, k=1$, (C): behavior in $\widetilde{Z}$.}
     \label{dim3_atrat_n}
 \end{figure}

 \begin{figure}[h!]
     \centering
     \includegraphics[width=14cm]{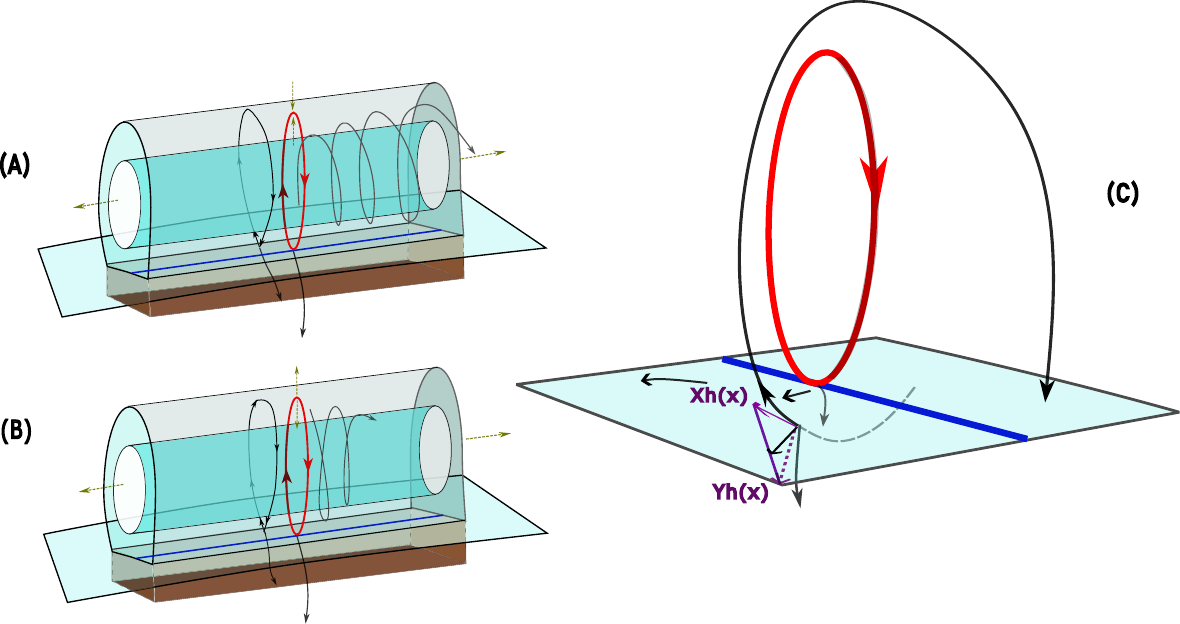}
     \put(-127,65){\footnotesize $x$}
    \put(-37,68){\footnotesize $\Sigma^c$}
    \put(-200,75){\footnotesize $\Sigma^e$}
    \put(-58,49){\footnotesize $\mathcal{S}_X^+$}
     \caption{(A): $Yh(p)<0, |\lambda |<1, k=1$, (B): $Yh(p)<0, |\lambda|>1, k=2$, (C): behavior in $\widetilde{Z}$.}
     \label{dim3_repel_n}
 \end{figure}

We can now describe the sets $Inv^\pm (N)$ and calculate the Conley index for each case, see Figures \ref{dim3_index_p}, \ref{dim3_index_n}.
\begin{itemize}
    \item $Yh(p)>0$:
       \begin{figure}[h!]
       \centering
       \includegraphics[width=13cm]{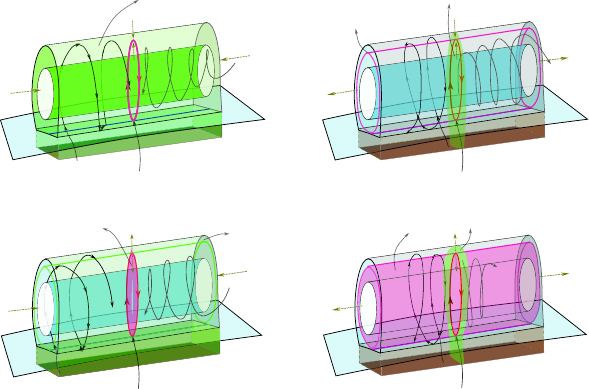}
       \put(-280,245){\footnotesize $Inv^+(N)$}
       \put(-100,245){\footnotesize $Inv^+(N)$}
       \put(-222,95){\footnotesize $Inv^+(N)$}
       \put(-85,105){\footnotesize $Inv^+(N)$}
       \put(-160,230){\footnotesize $Inv^-(N)$}
       \put(-341,100){\footnotesize $Inv^-(N)$}
       \put(-128,100){\footnotesize $Inv^-(N)$}
       \put(-320,130){\scriptsize $|\lambda|<1 \mbox{ and }k=0$}
       \put(-120,130){\scriptsize $|\lambda|<1 \mbox{ and }k=1$}
       \put(-320,-8){\scriptsize $|\lambda|>1 \mbox{ and }k=1$}
       \put(-120,-8){\scriptsize $|\lambda|>1 \mbox{ and }k=2$}
       \caption{Description of $Inv^\pm (N)$ for $Yh(p)>0$ and $k=0,1,2.$}
       \label{dim3_index_p}
   \end{figure}
    \begin{itemize}
        \item \textbf{ Case $|\lambda|<1$:} In this case $\gamma$ is attracting in the transverse direction associated to the Poincaré map.
        \begin{itemize}
            \item If $k=0$ (\autoref{dim3_atrat_p}, (A)), then $\gamma$ is an attracting periodic orbit and locally the sliding vector field $\widetilde{Z}$ converges to $p$. Hence $Inv^-(N)=\gamma$ and $Inv^+(N)=N$. 
                The index pair satisfies $P\simeq S^1$, $Q=\emptyset$, and therefore
                \[
                    \mathcal{C}(\gamma)\cong H_*( S^1)\cong \widetilde{H}_*(S^1\vee S^0).
                \] 
            \item If $k=1$ (\autoref{dim3_repel_p}, (A)), then $\gamma$ has one unstable direction in relation to the vector field $X$ and the sliding vector field has at least one trajectory that converges to $p$. In this case from the classical theory, there exists a cylinder of type $\gamma \times I$ such that the trajectory in negative time for all points in the cylinder tends to $\gamma$. Since $\gamma$ is tangent to $\Sigma$ and the analyses is local, sufficiently close to $\gamma$ the set $Inv^-(N)$ is topologically equivalent to $S^1 \times I$. Again from the classical theory, since $\gamma$ has a saddle type behavior, $Inv^+(N)\simeq D^2 \cap N$ since we have proved that the set $Inv^+(N)$ has at least one trajectory in $\widetilde{Z}$. Then  
                \[
                \mathcal{C}(\gamma)\cong H_*(P,Q)\cong \widetilde{H}_*(P/Q)\cong  \widetilde H_*(S^1\vee S^2).
                \]
                where $Q=S^1 \sqcup S^1$.
        \end{itemize}
        \item \textbf{Case $|\lambda|>1$:}
        Now $\gamma$ is repelling in the transverse direction.

        \begin{itemize}
            \item If $k=1$ (\autoref{dim3_atrat_p}, (B)), then $Inv^+(N)\cong S^1 \times D^1$ a solid cylinder in $N$ which outer boundary coincides with the outer boundary of $N$ and the inner boundary is given by $\simeq \gamma \times D^1$, $Inv^-(N)\cong D^2 \times N$ the region inside $\gamma$, $Q\cong S^1$. See \autoref{dim3_index_p}.
            \[
               \mathcal{C}(\gamma)\cong H_*(S^1\times D^1,S^1)\cong \widetilde{H}_*(D^2).
            \]

            \item If $k=2$ (\autoref{dim3_repel_p}, (B)), then the sliding vector field has at least one trajectory that converges to $p$ in positive time. $Inv^+(N)\simeq D^2 \cap N$ the outer ring of $\gamma$, $Inv^-(N)\simeq S^1 \times D^1 \times I$ the solid cylinder, with outer cylinder boundary as $\gamma \times D^1$. In this case, $Inv^-(N) \cap \partial N$ is topologically equivalent to the boundary of $N$ minus a circle, then $Q\simeq S^1 \times D^1$ and
            \[
            \mathcal{C}(\gamma)\cong H_*(S^1\times D^1 \times I,S^1\times D^1)\cong \widetilde{H}_*(D^2).
            \]
        \end{itemize}

    \end{itemize}

    \item \textbf{Case $Yh(p)<0$.} In this case $\Sigma^{s/e}$ is an escaping region. The argument is analogous, reversing time. Now every point in the outside region does not belong to $Inv^+(N)$. Indeed, by the discussion above, at least one trajectory of the sliding vector field belongs to $Inv^-(N)$ near $p$ (see Figures \ref{dim3_atrat_n}, \ref{dim3_repel_n}). Repeating the same reasoning with $Inv^+$ and $Inv^-$ interchanged, we obtain:

\begin{figure}[H]
        \centering
        \includegraphics[width=13cm]{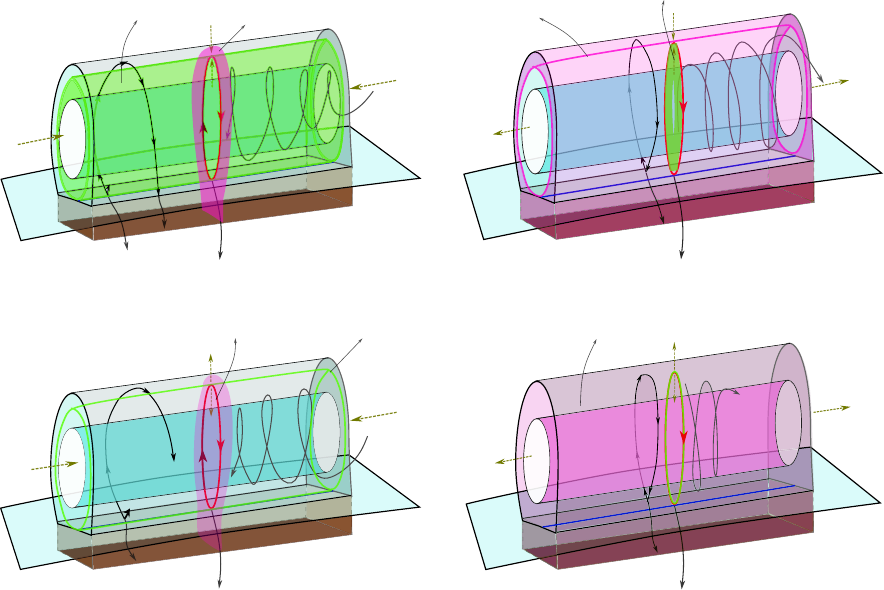}
        \put(-320,240){\footnotesize $Inv^+(N)$}
        \put(-270,240){\footnotesize $Inv^-(N)$}
       \put(-100,250){\footnotesize $Inv^+(N)$}
       \put(-217,107){\footnotesize $Inv^+(N)$}
       \put(-180,240){\footnotesize $Inv^-(N)$}
       \put(-278,107){\footnotesize $Inv^-(N)$}
       \put(-128,108){\footnotesize $Inv^-(N)$}
        \put(-320,130){\scriptsize $|\lambda|<1 \mbox{ and }k=0$}
       \put(-120,130){\scriptsize $|\lambda|<1 \mbox{ and }k=1$}
       \put(-320,-8){\scriptsize $|\lambda|>1 \mbox{ and }k=1$}
       \put(-120,-8){\scriptsize $|\lambda|>1 \mbox{ and }k=2$}
        \caption{Description of $Inv^\pm (N)$ for $Yh(p)<0$ and $k=0,1,2.$}
        \label{dim3_index_n}
    \end{figure}

    \begin{itemize}
        \item \textbf{Case $|\lambda|<1$:} 
        \begin{itemize}
            \item If $k=0$, $Inv^+(N)\simeq \gamma \times D^1 \times I$ the solid cylinder ``inside'' $\gamma$, $Inv^-(N)\simeq D^2 \cap N$ outer $\gamma$ and $Q\simeq S^1$
                \[
                    \mathcal{C}(\gamma)\cong H_*(S^1\times D^1,S^1)\cong \widetilde{H}_*(D^2) .
                \] 
            \item If $k=1$, $Inv^+(N)\simeq D^2 \cap N$ inside $\gamma$, $Inv^-(N)\simeq \gamma \times D^1 \times I$ the solid cylinder outside $\gamma$, $Q\simeq \gamma\times D^1$
                \[
                \mathcal{C}(\gamma)\cong H_*(S^1 \times D^1 \times I, S^1 \times D^1)\cong \widetilde{H}_*(D^2).
                \]
        \end{itemize}
            
        \item \textbf{Case $|\lambda|>1$:}
        \begin{itemize}
            \item If $k=1$, $Inv^+(N)\simeq \gamma \times D^1$, $Inv^-(N)\simeq D^2 \cap N$, $Q\simeq S^1 \sqcup S^1$
            \[
               \mathcal{C}(\gamma)\cong H_*(S^1\times D_1, S^1 \sqcup S^1) \cong \widetilde{H}_*(S^2 \vee S^1).
            \]

            \item If $k=2$, $Inv^+(N)=\gamma$, $Inv^-(N)=N\simeq $, $Q\simeq \partial N$.
            \[
            \mathcal{C}(\gamma)\cong H_*(S^1 \times D^2,S^1\times S^1)\cong \widetilde{H}_*(S^3 \vee S^2).
            \]
        \end{itemize}
    \end{itemize}

Hence, we have proved that
    \[
      \mathcal{C}(\gamma)=
        \begin{cases}
        \widetilde H_*(S^k\vee S^{k+1}), 
        & (|\lambda|<1,\,Yh(p)>0)\ \text{or}\ (|\lambda|>1,\,Yh(p)<0),\\[1ex]
        \widetilde{H}_*(*), 
        & (|\lambda|<1,\,Yh(p)<0)\ \text{or}\ (|\lambda|>1,\,Yh(p)>0).
        \end{cases}
    \]
\end{itemize}


\end{proof}

\subsection{Special Cases in Planar Filippov Vector Fields}\label{subsec:special}
In this section, we analyze particular configurations of limit cycles in planar Filippov vector fields.
\subsubsection{Figure-Eight-Shaped Limit Cycle}

    \begin{figure}[h!]
    \centering
    \includegraphics[width=15cm]{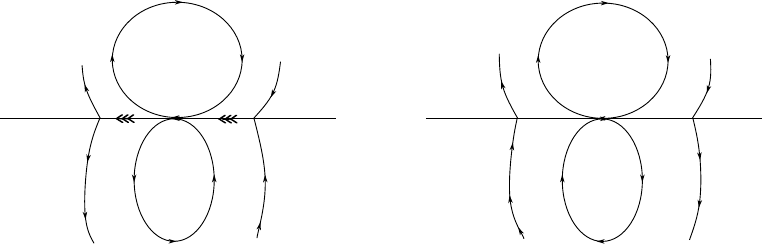}
    \put(-330,62){\footnotesize $p$}
    \put(-400,58){\footnotesize $\Sigma^e$}
    \put(-270,58){\footnotesize $\Sigma ^s$}
    \put(-290,120){\footnotesize $\gamma_X$}
    \put(-310,10){\footnotesize $\gamma_Y$}
    \put(-92,62){\footnotesize $p$}
    \put(-160,58){\footnotesize $\Sigma^c$}
    \put(-30,58){\footnotesize $\Sigma ^c$}
    \put(-50,120){\footnotesize $\gamma_X$}
    \put(-70,10){\footnotesize $\gamma_Y$}
    \caption{Figure-eight-shaped periodic orbit.}
    \label{foldfold}
\end{figure}

Let $Z=(X,Y)$ be a planar Filippov vector field. Assume there exist two hyperbolic limit cycles $\gamma_X$ and $\gamma_Y$, associated with the smooth vector fields $X$ and $Y$, respectively, both tangent to the switching manifold $\Sigma$ at the same point, see \autoref{foldfold}. More precisely,
\[
\gamma_X \cap \Sigma = \gamma_Y \cap \Sigma = \{p\},
\]
where $p$ is a visible fold--fold point. Define
\[
\gamma := \gamma_X \cup \gamma_Y.
\]

Since $p$ is a visible fold--fold point, we have
\[
Xh(p)=Yh(p)=0, \qquad X^2h(p)>0, \qquad Y^2h(p)<0.
\]
Consequently, for $x \in \Sigma$ sufficiently close to $p$, two distinct configurations may occur:
\begin{itemize}
    \item $Xh(x)Yh(x)<0$;
    \item $Xh(x)Yh(x)>0$.
\end{itemize}

Throughout this discussion, we consider only isolated orbits in the sense of Conley; in particular, no singularities of the Filippov vector field are present in a neighborhood of $p$.

In the first situation, $\gamma$ is not isolated in the set of closed orbits. Indeed, since $Xh(x)Yh(x)<0$ for all $x \in \Sigma \cap B_\varepsilon(p)$ and $Xh(p)=Yh(p)=0$, the switching manifold locally splits into a sliding region and an escaping region. In these regions, the Filippov vector field is defined through the closed convex combination of $X$ and $Y$. The associated sliding vector field provides a connection between the sliding and escaping sectors near $p$, see \autoref{foldfold_1}.

    \begin{figure}[h!]
    \centering
    \includegraphics[width=15cm]{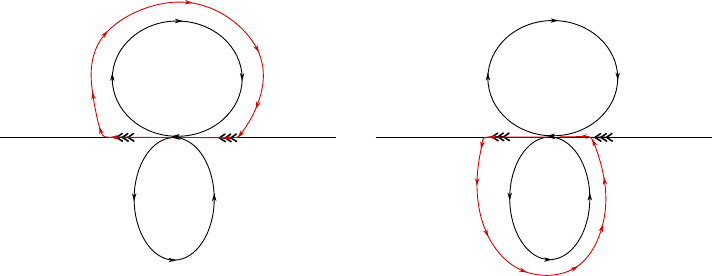}
    \put(-325,75){\footnotesize $p$}
    \put(-380,70){\footnotesize $\Sigma^e$}
    \put(-280,70){\footnotesize $\Sigma ^s$}
    \put(-98,75){\footnotesize $p$}
    \put(-160,70){\footnotesize $\Sigma^e$}
    \put(-50,70){\footnotesize $\Sigma ^s$}
    \caption{Non-isolated figure-eight-shaped periodic orbit.}
    \label{foldfold_1}
    \end{figure}
    
Therefore, $\gamma$ is not isolated among closed orbits and does not constitute an isolated invariant set in the sense of Conley.

If $Xh(x)Yh(x)>0$ in a neighborhood of $p$, then $\Sigma$ consists locally only of crossing regions. In this case, the Filippov vector field coincides with the classical concatenation of the smooth flows of $X$ and $Y$, except at the point $p$.

Since $p$ is a visible fold--fold point, both vector fields are tangent to $\Sigma$ at $p$ and cross the switching manifold in opposite directions on each side of $p$. Under this condition, a figure-eight-shaped isolated limit cycle \textit{may} occur.

Let $\varepsilon>0$ be sufficiently small and define
\[
\alpha_\varepsilon := (\Sigma \cap B_\varepsilon(p)) \setminus \{p\}.
\]
Take $x_0 \in \alpha_\varepsilon$ and assume, without loss of generality, that $Xh(x_0)>0$. Since $\gamma_X$ is a limit cycle of $X$, there exists a smallest $t_{x_0}>0$ such that
\[
\varphi_X(t_{x_0},x_0)\in\Sigma
\quad \text{and} \quad
Xh(\varphi_X(t_{x_0},x_0))<0.
\]
Set $x_1:=\varphi_X(t_{x_0},x_0)$. Then $Yh(x_1)<0$. Because $\gamma_Y$ is a limit cycle of $Y$, there exists a smallest $t_{x_1}>0$ such that
\[
\varphi_Y(t_{x_1},x_1)\in\Sigma
\quad \text{and} \quad
Yh(\varphi_Y(t_{x_1},x_1))>0.
\]
\begin{figure}
    \centering
    \includegraphics[width=7cm]{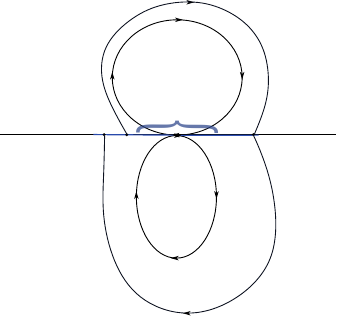}
    \put(-98,98){\footnotesize $p$}
    \put(-98,118){\footnotesize $\alpha_\epsilon$}
    \put(-127,100){\footnotesize $x_0$}
    \put(-45,100){\footnotesize $x_1$}
    \put(-175,110){\footnotesize $\varphi_Y(t_{x_1},x_1)$}
    \put(2,105){\footnotesize $\Sigma$}
    \caption{First return map near $\gamma$.}
    \label{folfold_3}
\end{figure}

In the crossing region, the Filippov flow is defined as the concatenation of the flows of $X$ and $Y$. Thus, for each $x\in\alpha_\varepsilon$, there exists a smallest $t_x>0$ such that
\[
\varphi_Z(t_x,x)\in\Sigma
\quad \text{and} \quad
\operatorname{sgn}(Xh(\varphi_Z(t_x,x)))=\operatorname{sgn}(Xh(x)).
\]

Since $\Sigma$ is one-dimensional, define
\[
\tau:\alpha_\varepsilon^+ \to \mathbb{R},
\qquad
\tau(x)=|x-\varphi_Z(t_x,x)|,
\]
where
\[
\alpha_\varepsilon^+:=\{x\in\alpha_\varepsilon: Xh(x)>0\},
\]
see \autoref{folfold_3}.

If $\tau(x)\neq 0$ for all $x\in\alpha_\varepsilon^+$ and $\epsilon>0$, then the figure-eight orbit is isolated, see \autoref{foldfold_2}.

\begin{figure}[h!]
    \centering
    \includegraphics[width=13cm]{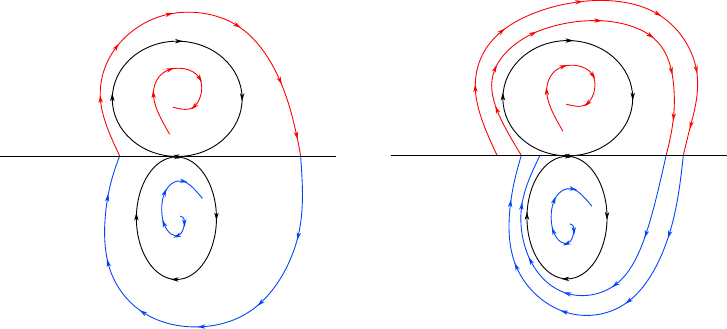}
    \put(-283,82){\footnotesize $p$}
    \put(-83,82){\footnotesize $p$}
    \put(-197,85){\footnotesize $\Sigma$}
    \put(0,85){\footnotesize $\Sigma$}
    \caption{Examples where $\tau(x)= 0$ and $\tau(x)\neq 0$, respectively.}
    \label{foldfold_2}
    \end{figure}

\begin{proposition}
    Let $\gamma=\gamma_X\cup \gamma_Y$ be a figure-eight-shaped \textbf{isolated} limit cycle of the planar Filippov vector field $Z$. Let $k_X$ and $k_Y$ be the morse index of $\gamma_X$ and $\gamma_Y$, resp.

Then, the Conley index of $\gamma$ is:
\begin{itemize}
    \item  $\mathcal{C}(\gamma)=H_*(S^1\vee S^1)$, if $k_X=k_Y=0$;
    \item $\mathcal{C}(\gamma)\cong \widetilde{H}_*(S^2 \vee S^1 \vee S^1)$, if $k_X=k_Y=1$;
    \item $\mathcal{C}(\gamma)\cong \widetilde{H}_*(S^1)$, if $k_X \ne k_Y$.
\end{itemize}

\end{proposition}
\begin{proof}  

Since $\gamma$ is isolated, there exists $\epsilon>0$ such that $\tau(x)\ne 0 \forall x \in \alpha_\epsilon^+$.
We choose a small tubular neighborhood $N$ of $\gamma$ such that:
\begin{enumerate}
    \item $N$ contains no other invariant sets except $\gamma$;
    \item the intersection $N \cap \Sigma$ is a neighborhood of $p$ in $\Sigma$.
\end{enumerate}
\paragraph{Construction of the isolating neighborhood:}
With the same ideas as in the previous section, we want to construct an isolating neighborhood that ``respect'' the flow.

Let $q \in \alpha_\epsilon$ and, without loss of generality, assume $Xh(q)<0$. There exists $t_0,t_1>0$ such that $q_0:=\varphi_X(-t_0,q)\in \Sigma$ and $q_1:=\varphi_Y(t_1,q) \in \Sigma$ since $q$ is sufficiently close to $p$, see \autoref{N_fig_oito}.

\begin{figure}[h!]
    \centering
    \includegraphics[width=11cm]{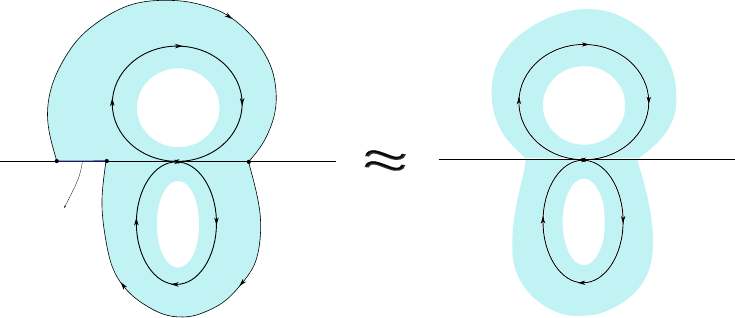}
    \put(-205,61){\small $q$}
    \put(-292,61){\small $q_0$}
    \put(-272,70){\small $q_1$}
    \put(-240,61){\footnotesize $p$}
    \put(-210,127){\small $\alpha_0$}
    \put(-215,7){\small $\alpha_1$}
    \put(-290,40){\small $I$}
    \put(-170,65){\small $\Sigma$}
    \put(-67,61){\footnotesize $p$}
    \put(0,65){\small $\Sigma$}
    \caption{Isolating neighborhood $N$ for an figure-eight-shaped limit cycle.}
    \label{N_fig_oito}
\end{figure}

Let $$\alpha_0:=\{x \in \mathbb{R}^2 s.t. \mbox{ } \exists \mbox{ }t \in [0,t_0] ; x=\varphi_X(-t,q)\}$$ and $$\alpha_1:=\{x \in \mathbb{R}^2 s.t. \mbox{ }\exists \mbox{ } t \in [0,t_1]; x=\varphi_Y(t,q)\}.$$
Denote by $\mathcal{\overline{A}}$ the region delimited by $\alpha_0 \cup \alpha_1\cup I$ where $I\subset \Sigma$ is the segment between $\varphi_X(-t_0,q)$ and $\varphi_Y(t_1,q)$. As the previous section, we can consider the cut round handle for $\gamma_X$ and $\gamma_Y$ as $N_X$ and $N_Y$ respectively. Then we define $N$ as the intersection of $\mathcal{\overline{A}}$ with $N_X$ and $N_Y$. Hence $N$ is well defined by construction, compact and since $\tau(x)\ne 0$, $\gamma$ is the maximal invariant set in $N$.

Observe that $N$ can be simply expressed as $$N \cong (S^1 \vee S^1)\times D^1.$$

\paragraph{Behavior outside $\gamma$ and the Conley index:}
\begin{itemize}
    \item Case $k_X=k_Y=0$:

     \begin{figure}[h!]
                \centering
                \includegraphics[width=6cm]{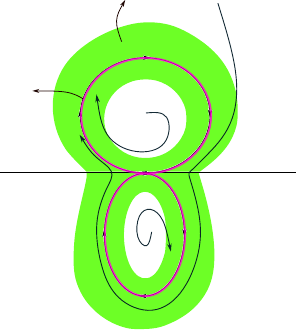}
                \put(2,87){\small $\Sigma$}
                \put(-100,195){\small $Inv^+(N)$}
                \put(-195,133){\small $Inv^-(N)$}
                \caption{Description of the invariant sets $Inv^\pm (N)$ for the case $k_X=k_Y=0$.}
                \label{fig8_1}
 \end{figure}
 
    In this case, both cycles $\gamma_X$ and $\gamma_Y$ are attracting for the smooth vector fields $X$ and $Y$, respectively. Hence, $|x|>|\varphi_Z(t_x,x)|$ for $x \in \Sigma \cap N$ and $\gamma$ is an attracting orbit which implies that $Inv^-(N)=\gamma$ and $Inv^+(N)=N$, see \autoref{fig8_1}. Therefore, an index pair may be chosen with $P\simeq S^1\vee S^1 $ and $Q=\emptyset$. Then the Conley index is given by
\[
\mathcal{C}(\gamma) \cong H_*(S^1 \vee S^1).
\]

    \item  Case $k_X=k_Y=1$: 
    \begin{figure}[H]
        \centering
        \includegraphics[width=6cm]{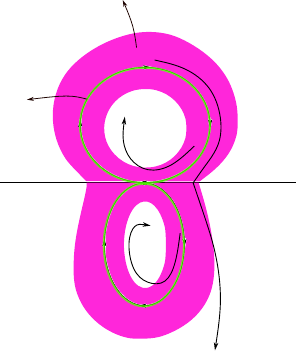}
        \put(2,94){\small $\Sigma$} 
        \put(-110,205){\small $Inv^-(N)$}
        \put(-195,141){\small $Inv^+(N)$}
        \caption{Description of the invariant sets $Inv^\pm (N)$ for the case $k_X=k_Y=1$.}
        \label{fig8_2}
\end{figure}
  When both limit cycles $\gamma_X$ and $\gamma_Y$ are repelling, then $|x|<|\varphi_Z(t_x,x)|$ for $x \in \Sigma \cap N$ and $\gamma$ is repelling. In this case, $Inv^-(N)=N$ and $Inv^+(N)=\gamma$, see \autoref{fig8_2}.
Then $P\simeq N$ and $Q\simeq \partial P$ which implies $$\mathcal{C}(\gamma)=H_*(P,Q)\cong \widetilde{H}_*(P/Q)= \widetilde{H}_*(S^2 \vee S^1 \vee S^1),$$ since $(P,Q)$ is a good pair.

        \item Case $k_X \ne k_Y$: 

Assume without loss of generality that $\gamma_X$ is attracting ($k_X=0$) and $\gamma_Y$ is repelling ($k_Y=1$). The opposite configuration is analogous.

\begin{figure}[h!]
         \centering
         \includegraphics[width=14cm]{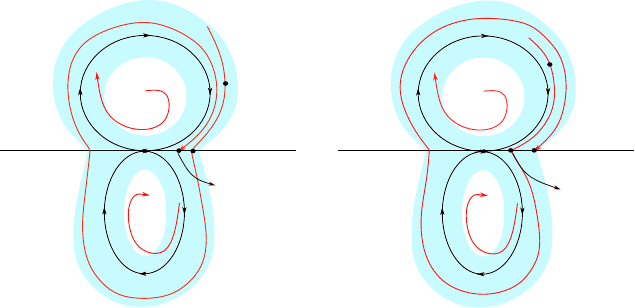}
         \put(2,95){\small $\Sigma$}
         \put(-210,95){\small $\Sigma$}
         \put(-254,140){\small $x_0$}
         \put(-47,151){\small $x_0$}
         \put(-273,93){\small $x_1$}
         \put(-45,73){\small $x_1$}
         \put(-260,75){\small $\varphi_Z(t_{x_1},x_1)$}
         \put(-65,92){\small $\varphi_Z(t_{x_1},x_1)$}
         \caption{Attracting and repelling dynamics near $\gamma$, respectively.}
         \label{fig_5}
     \end{figure}
Consider a trajectory with initial condition $x_0\in N$ near $\gamma_X$. Then there exists $t_1>0$ such that $x_1:=\varphi_X(t_1,x_0) \in \Sigma$. Now, by discussion above, there exists $t_{x_1}>0$ that defines the return map $\varphi_Z(t_{x_1},x_1) \in \Sigma$.
    Depending on the dynamics, near $\gamma$ one of two behaviors may occur (see \autoref{fig_5}):
    \begin{enumerate}
        \item \textbf{Concatenated attracting behavior:} The trajectory moves away from $\gamma_Y$ but returns to $\Sigma$ in a neighborhood of $p$ such that $|x_1|>|\varphi_Z(t_{x_1},x_1)|$. Then the figure-eight cycle acts as an attracting cycle. In this case .
        \item \textbf{Concatenated repelling behavior:} The trajectory evolves along the orbit $\gamma_Y$ under the flow generated by $\gamma_Y$ and eventually exits $N$ through the outer boundary, so $\gamma$ is locally repelling. Equivalently, the associated return map is expansive near $\gamma$ and $|x_1|<|\varphi_Z(t_{x_1},x_1)|$.
    \end{enumerate}

Let $(P,Q)$ be an index pair for $\gamma$ contained in $N$:
\begin{itemize}
    \item $P$ deformation retracts onto $N$ as a compact neighborhood of $Inv^-(N)$, and hence onto the figure-eight $S^1 \vee S^1$;
    \item $Q\cong\partial P \cap \partial N$ .
\end{itemize}

- In the attracting scenario, $Inv^-(N)=\gamma \cup U$ where $U\subset N$ is the region of $N$ inside $\gamma_Y$, see \autoref{fig8_3}. Then, $Q$ represents the boundary of this region, $\partial U$, and since $(P,Q)$ is a good pair, we can take the quotient space. Therefore, $Q$ contracts only part of the figure-eight, leaving one circle homotopically unchanged:
\[
P/Q \simeq S^1.
\]
\begin{figure}[h!]
    \centering
    \includegraphics[width=6cm]{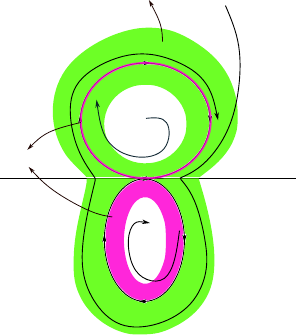}
    \put(2,87){\small $\Sigma$}
    \put(-98,196){\small $Inv^+(N)$}
    \put(-180,99){\small $Inv^-(N)$}
    \caption{Description of the invariant sets $Inv^\pm (N)$ for the attracting behavior, where $k_X=0$ and $k_Y=1$.}
    \label{fig8_3}
\end{figure}

- In the repelling scenario, $Inv^-(N)=N \setminus V$ where $V \subset N$ is the region of $N$ inside $\gamma_X$, see \autoref{fig8_4}.

\begin{figure}[h!]
    \centering
    \includegraphics[width=6cm]{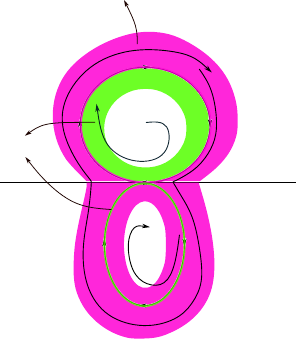}
    \put(2,87){\small $\Sigma$}
    \put(-112,196){\small $Inv^-(N)$}
    \put(-180,108){\small $Inv^+(N)$}
    \caption{Description of the invariant sets $Inv^\pm (N)$ for the repelling behavior, where $k_X=0$ and $k_Y=1$.}
    \label{fig8_4}
\end{figure}

Then, since $Q\cong\partial P \cap \partial N$, we can look for the quotient space and 
\[
P/Q\simeq [(S^2\setminus D^1),\{x,y\}] \simeq S^1.
\]

\paragraph{Opposite configuration.}
If $\gamma_X$ is repelling and $\gamma_Y$ is attracting, the argument is analogous: concatenated flows may either trap trajectories near the attracting branch or allow them to escape along the repelling branch. In both cases, the Conley index is
\[
\mathcal{C}(\gamma) \cong \widetilde{H}_*(S^1).
\]

\end{itemize}
\end{proof}

\subsubsection{Linear Hybrid Cycles}

In this subsection, we study hybrid limit cycles in linear planar Filippov vector fields. These are periodic orbits composed by trajectories of the regular vector fields $X$ and/or $Y$ together with trajectories of the sliding vector field defined on the switching manifold $\Sigma$. Such cycles combine crossing, sliding and escaping dynamics, providing a natural extension of classical periodic orbits in nonsmooth systems.

Since both vector fields $X$ and $Y$ are linear, each admits at most one tangency point with the switching manifold. Moreover, in this section we restrict our attention to when the tangency points are actually regular tangency points. As a consequence, there can be at most one sliding or escaping region, which implies that the system admits at most one hybrid cycle. According to the structure of the orbit, we distinguish two classes of hybrid limit cycles, see \autoref{hybrid}.

\begin{figure}[H]
    \centering
    \includegraphics[width=14cm]{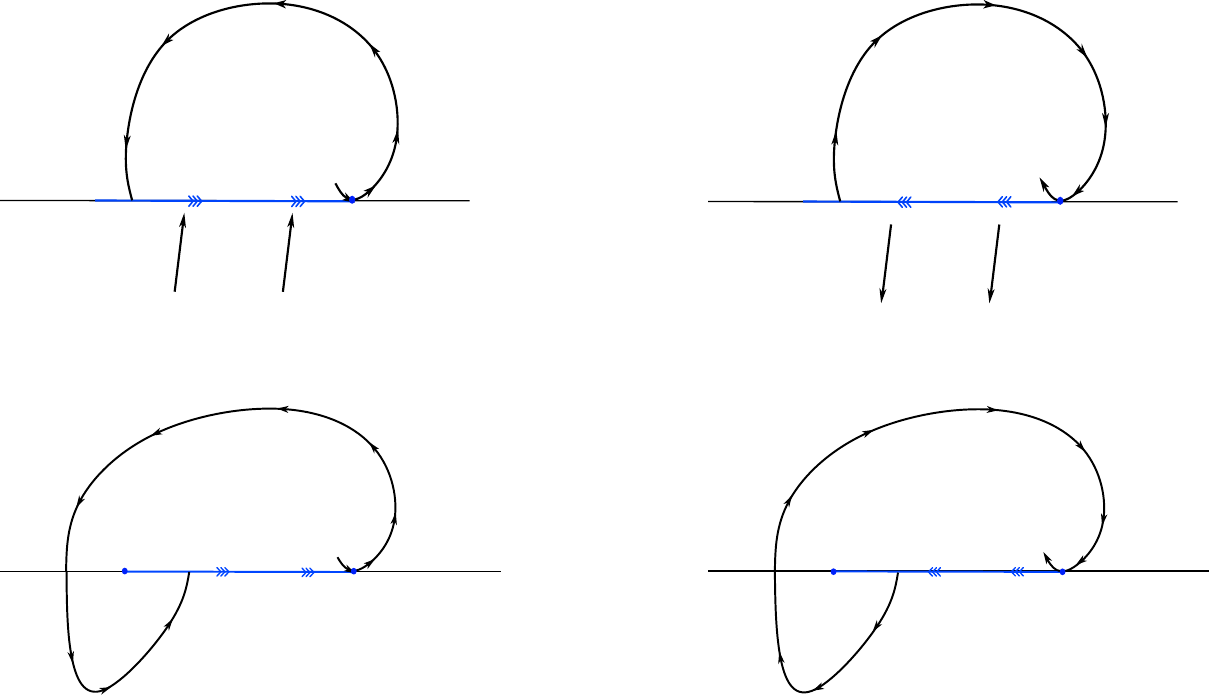}
    \put(-410,160){\small $\Sigma$}
    \put(-410,38){\small $\Sigma$}
    \put(-178,160){\small $\Sigma$}
    \put(-178,38){\small $\Sigma$}
    \put(-400,200){\small $X$}
    \put(-400,80){\small $X$}
    \put(-168,200){\small $X$}
    \put(-168,80){\small $X$}
    \put(-400,140){\small $Y$}
    \put(-400,0){\small $Y$}
    \put(-168,140){\small $Y$}
    \put(-168,0){\small $Y$}
    \put(-285,155){\small $p$}
    \put(-52,155){\small $p$}
    \put(-285,35){\small $p$}
    \put(-52,35){\small $p$}
    \put(-220,180){\small Type I}
    \put(-220,60){\small Type II}
    \caption{Examples of type I and type II cycles.}
    \label{hybrid}
\end{figure}

\begin{itemize}
    \item A hybrid limit cycle $\gamma$ is said to be of \textbf{type I} if $\gamma$ is composed of trajectories of the sliding vector field in $\Sigma^{e,s}$ together with trajectories of only one regular vector field, either $X$ or $Y$.

    \item A hybrid limit cycle $\gamma$ is said to be of \textbf{type II} if $\gamma$ is composed of trajectories of the sliding vector field in $\Sigma^{e,s}$ together with trajectories of both regular vector fields $X$ and $Y$.
\end{itemize}

We now justify that hybrid limit cycles of both types can be regarded as isolated invariant sets. The analysis depends on whether the cycle contains a sliding region or an escaping region.

Let $N$ be a compact tubular neighborhood of $\gamma$ sufficiently close to it, that is, a compact annular neighborhood diffeomorphic to $S^1\times[-1,1]$ and containing $\gamma$ in its interior. Without loss of generality, assume that $p$ is a visible tangency point of the vector field $X$, that is, $Xh(p)=0$, $X^2h(p)>0$ and $Yh(p)\ne 0$.

\textbf{Sliding case.} Assume first that $\gamma$ contains a sliding segment of $\Sigma^s$, and that $Yh(p)>0$. Let $x \in N$ be sufficiently close to $\gamma$. If $\gamma$ is of type I, then for any $x \in X \cap N$, the positive trajectory of $x$ approaches the switching manifold and eventually reaches the sliding segment, which belongs to $\gamma$.

In the case of a type II cycle, the trajectory first crosses the switching manifold and follows the flow generated by the vector field $Y$ before concatenating with the sliding dynamics and subsequently entering $\gamma$. If $x \in Y \cap N$, the trajectory initially evolves according to $Y$ and eventually reaches the sliding segment.

Consequently, every forward trajectory starting in $N$ remains in $N$ for all positive time and approaches $\gamma$. Therefore,
\[
Inv^{+}(N)=N.
\]

On the other hand, the set $Inv^-(N)$ is precisely the cycle $\gamma$. Indeed, since $p$ is a visible tangency point of the vector field $X$ and the trajectory returns to $\Sigma$, there exists an invariant set $S$ contained in the region of $X$ delimited by $\gamma$. Furthermore, for any point $x \in X \cap N \setminus \{\gamma\}$ in the region bounded by $\gamma$, the $\alpha$-limit set satisfies $\alpha(x) \subset S$.

 \begin{figure}[H]
     \centering
     \includegraphics[width=14cm]{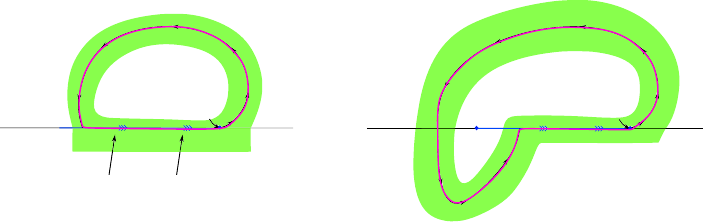}
     \put(-230,50){\small $\Sigma$}
     \put(2,50){\small $\Sigma$}
     \put(-275,47){\small $p$}
     \put(-46,47){\small $p$}
     \put(-230,90){\small $Inv^+(N)=N$}
     \put(-230,80){\small $Inv^-(N)=\gamma$}
     \caption{Description of the sets $Inv^\pm (N)$ for the sliding case.}
     \label{hyb_1}
 \end{figure}
 
If $x \in X \cap N \setminus \{\gamma\}$ lies outside the region bounded by $\gamma$, then, in backward time, the trajectory reaches the crossing region and follows the dynamics generated by the vector field $Y$, without entering the sliding region.
Thus, no point $x\in N\setminus \gamma$ admits a complete negative orbit entirely contained in $N$. The only points whose admits trajectories that remain in $N$ for all negative time are the points of the periodic orbit itself. Therefore,
\[
Inv^{-}(N)=\gamma.
\]
 Moreover, $\gamma$ is a locally attracting cycle, see \autoref{hyb_1}.
 
\noindent
\textbf{Escaping case.}

 \begin{figure}[H]
     \centering
     \includegraphics[width=14cm]{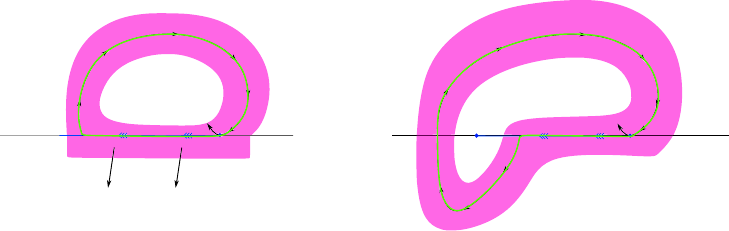}
     \put(-235,48){\small $\Sigma$}
     \put(2,48){\small $\Sigma$}
     \put(-279,45){\small $p$}
     \put(-57,45){\small $p$}
     \put(-230,90){\small $Inv^-(N)=N$}
     \put(-230,80){\small $Inv^+(N)=\gamma$}
     \caption{Description of the sets $Inv^\pm (N)$ for the escaping case.}
     \label{hyb_2}
 \end{figure}

Suppose now that $\gamma$ contains an escaping segment of $\Sigma$. In this case, the dynamics is obtained from the sliding case by time reversal. Consequently, the hybrid cycle $\gamma$ is an repelling isolated invariant set, see \autoref{hyb_2}. More precisely, there exists a neighborhood $N$ of $\gamma$ such that
\[
Inv^{+}(N)=\gamma \quad \text{and} \quad Inv^{-}(N)=N.
\]

\begin{proposition}
Let $\gamma$ be a hybrid limit cycle of a linear planar Filippov vector field and assume that every tangency point $p$ is a regular tangency point.

\begin{itemize}
    \item If $\gamma$ contains a sliding segment, then $\gamma$ is an isolated invariant set satisfying
    \[
    Inv^{+}(N)=N,
    \qquad
    Inv^{-}(N)=\gamma,
    \]
    for a sufficiently small isolating neighborhood $N$. In particular, $\gamma$ is locally attracting in forward time, and its Conley index is
    \[
    \mathcal{C}(\gamma)=H_*(S^1).
    \]

    \item If $\gamma$ contains an escaping segment, then $\gamma$ is an isolated invariant set satisfying
    \[
    Inv^{-}(N)=N,
    \qquad
    Inv^{+}(N)=\gamma,
    \]
    for a sufficiently small isolating neighborhood $N$. In particular, $\gamma$ is locally repelling in forward time, and its Conley index is
    \[
    \mathcal{C}(\gamma)= \widetilde{H}_*(S^2\vee S^1).
    \]
\end{itemize}
\end{proposition}

\begin{proof}
The proof of the isolated invariant property follows from the previous discussion.

In the sliding case, $P=Inv^-(N)$ and $Q= \emptyset$ is an index pair for $\gamma$ according to the definition of the Conley index by Mrozek. Then,
\[
\mathcal{C}(\gamma)=H_*( S^1).
\]

In the escaping case, the time orientation is reversed. In this case, since $Inv^-(N)=N$, it follows that $(P,Q)=(N, \partial N)$ is an index pair for $\gamma$. Hence, $$\mathcal{C}(\gamma)=H_*(N,\partial N)\simeq \widetilde{H}_*(N/\partial N)$$

Since $N$ is homeomorphic to the annulus $S^1\times[0,1]$, collapsing both boundary components to a single point produces a space homotopy equivalent to
\[
S^2\vee S^1.
\]
Therefore,
\[
\mathcal{C}(\gamma)\simeq \widetilde{H}_*( S^2\vee S^1).
\]

\end{proof}

\section{Final Remarks}\label{sec:finalremarks}
Filippov systems provide a natural and widely used framework for modeling phenomena with abrupt changes. Despite their relevance, the global understanding of discontinuous vector fields remains incomplete. The presence of sliding motion, non-uniqueness of solutions, and mixed-type trajectories introduce significant challenges to classical approaches based on smooth flows.

In this study, we analyzed the local topological structure of periodic invariant sets by computing their Conley indices within the framework of multivalued dynamics. Although these results are local in nature, they contribute to a broader program in which invariant sets serve as building blocks for global dynamics.

A systematic description of these local invariants is expected to play an important role in understanding the global organization of Filippov systems. Its topological characterization may be useful in future studies of Morse decompositions, continuation phenomena, and global bifurcations in nonsmooth dynamics.

{%
\renewcommand{\bibname}{Referências}
}
\bibliography{bib.bib}
\bibliographystyle{acm}

\vspace{1cm}

{\small 
\noindent L. Cândido\\[-5pt]
Institute of Mathematics, Statistics and Scientific Computing (IMECC) \\[-5pt]
University of Campinas\\[-5pt]
R. Sérgio Buarque de Holanda, 651, Campinas - SP, 13083-859, BRAZIL \\[-5pt]
e-mail: leticia.souzacndd@gmail.com

\vspace{0.35cm}

\noindent D. V. S. Lima\\[-5pt]
Center of Mathematics, Computing and Cognition (CMCC)\\[-5pt]
Federal University of ABC\\[-5pt]
Av. dos Estados, 5001 - Santo André - SP - CEP: 09280-560, BRAZIL\\[-5pt]
e-mail: dahisy.lima@ufabc.edu.br

\vspace{0.35cm}

\noindent K. A. de Rezende\\[-5pt]
Institute of Mathematics, Statistics and Scientific Computing (IMECC) \\[-5pt]
University of Campinas\\[-5pt]
R. Sérgio Buarque de Holanda, 651, Campinas - SP, 13083-859, BRAZIL \\[-5pt]
e-mail: ketty@unicamp.br
}
\end{document}